\documentclass[reqno,a4paper,12pt]{amsart}
\usepackage{latexsym, amsmath, amssymb, amsthm, a4}
\usepackage{graphicx}
\usepackage{xcolor} 
\usepackage{subcaption}

\usepackage{amsfonts, bbold, bbm, marvosym}

\usepackage{mathrsfs}
\usepackage[mathscr]{eucal}
\font\sc=rsfs10 at 12pt
\usepackage{yhmath}

\renewcommand{\a}{\alpha}
\newcommand{\am}{\pmb{\alpha}}
\newcommand{\be}{\beta}
\newcommand{\g}{\gamma}
\newcommand{\G}{\Gamma}
\newcommand{\de}{\delta}
\newcommand{\D}{\Delta}
\newcommand{\e}{\epsilon}
\newcommand{\ve}{\varepsilon}
\newcommand{\z}{\zeta}
\newcommand{\y}{\eta}

\newcommand{\io}{\iota}
\newcommand{\ka}{\kappa}
\newcommand{\vk}{\varkappa}
\newcommand{\la}{\lambda}

\newcommand{\m}{\mu}
\newcommand{\n}{\nu}
\newcommand{\x}{\xi}

\newcommand{\ro}{\rho}

\newcommand{\s}{\sigma}
\newcommand{\si}{\sigma}

\newcommand{\vs}{\varsigma}
\newcommand{\ta}{\tau}
\newcommand{\f}{\phi}

\newcommand{\F}{\Phi}
\newcommand{\vp}{\varpi}

\newcommand{\pk}{\pmb{\kappa}}
\newcommand{\om}{\omega}
\newcommand{\Om}{\Omega}

\newcommand{\U}{\Upsilon}

\newcommand{\R}{{\mathbb R}}

\newcommand{\Tt}{\mbox{\texttt{T}}}

\newcommand{\Ct} {\mbox{\texttt{C}}} 
 \newcommand{\Et}{\mbox{\texttt{E}}}
 \newcommand{\fT}{\tilde{f}}

 \newcommand{\Bt}{\mbox{\texttt{B}}} 
 
\newcommand{\ab}{{\mathbf a}}

\newcommand{\bb}{{\mathbf b}}

\newcommand{\cb}{{\mathbf c}}

\newcommand{\fb}{{\mathbf f}}

\newcommand{\hb}{{\mathbf h}}

\newcommand{\mb}{{\mathbf m}}
\newcommand{\nb}{{\mathbf n}}

\newcommand{\rb}{{\mathbf r}}

\newcommand{\tb}{{\mathbf t}}

\newcommand{\xb}{{\mathbf x}}
\newcommand{\yb}{{\mathbf y}}

\newcommand{\Bb}{{\mathbf B}}
\newcommand{\Cb}{{\mathbf C}}

\newcommand{\Eb}{{\mathbf E}}

\newcommand{\Gb}{{\mathbf G}}

\newcommand{\Ib}{{\mathbf I}}

\newcommand{\Kb}{{\mathbf K}}

\newcommand{\Nb}{{\mathbf N}}
\newcommand{\Ob}{{\mathbf O}}

\newcommand{\Qb}{{\mathbf Q}}

\newcommand{\Sbb}{{\mathbf S}}

\newcommand{\Ac}{{\mathcal A}} \newcommand{\Bc}{{\mathcal B}} 
 \newcommand{\Dc}{{\mathcal D}} 
  
 \newcommand{\Hc}{{\mathcal H}} 
\newcommand{\Ic}{{\mathcal I}} \newcommand{\Jc}{{\mathcal J}} 
 \newcommand{\Lc}{{\mathcal L}} 
 
\renewcommand{\Mc}{{\mathcal M}} 
  
 \newcommand{\Pc}{{\mathcal P}} 
 \newcommand{\Rc}{{\mathcal R}} 
\newcommand{\Sc}{{\mathcal S}}  
\newcommand{\Tc}{{\mathcal T}}  

 \newcommand{\Yc}{{\mathcal Y}}% NEW \newcommand{\Wc}{{\mathcal %W}} % 
 \newcommand{\Qc}{{\mathcal Q}}
\newcommand{\Fs}{\sc\mbox{F}\hspace{1.0pt}} 
\newcommand{\Gs}{\sc\mbox{G}\hspace{1.0pt}}

  \newcommand{\dist}{{\rm dist}\,}

\newcommand{\grad}{{\rm grad}\,}

\newcommand{\supp}{\hbox{{\rm supp}}\,}

 \newcommand{\jt}{\mathrm{j}}

\newcommand{\diam}{\operatorname{diam\,}}

\newcommand{\meas}{\operatorname{meas\,}}

\newcommand{\dn}{d_0}

\newtheorem{thm}{Theorem}[section]
\newtheorem{cor}[thm]{Corollary}
\newtheorem{lem}[thm]{Lemma}

\theoremstyle{definition}
\newtheorem{defin}[thm]{Definition}

\newtheorem{property}[thm]{Property}
\numberwithin{equation}{section}

\begin{document}

\author[]{Grigori Rozenblum}
\address{Chalmers University of Technology, Gothenburg, Sweden}
\email{grigori@chalmers.se}

\author[]{Nikolai A.  Shirokov}
\address{St. Petersburg State University, Univ.Naberezhnaya, 7/9, St.Petersburg 199034; Nation. Research Univ. Higher School of Economics, 
Kantemirovskaya Str.
3a, St. Petersburg, 194100;  Russia}
\email{nikolai.shirokov@gmail.com}

% Give an abbreviation of the title for the running page headers.
%\titlerunning{H{\"o}lder approximation}

% Here you can enter the full article title. 
\title[Higher Order H{\"o}lder approximation]{Higher Order H{\"o}lder approximation by 
solutions of  second order elliptic equations} 
\begin{abstract}For a given second order elliptic operation $\Lc$ 
in a domain 
$\Om\subset\R^\Nb$, $\Nb\ge3$, and a compact set 
$\Kb\subset\Om$, order $\Nb$-$2$-Ahlfors-David 
regular, we define the space $\Hc^{\rb+\om}_{\Lc}(\Kb)$ of 
continuous functions 
$f(x),\, x\in\Kb$, admitting, for any $\de>0$, a local 
approximation in the $\de
$-neighborhood of any point $x\in\Kb$, with 
$\de^{\rb}\om(\de)$-error estimate, 
by solutions of the equation $\Lc u=0$. For such functions, we 
prove the 
existence of a global approximation $v_\de$ on $\Kb$ with the 
same order  of 
error estimate, by a solution of the same equation in  a 
$\de$-neighborhood of 
$\Kb$. A number of properties of these functions $v_\de$ and 
their derivatives 
are established.
\end{abstract}
\maketitle
 \section{Introduction}\label{Intro}
  \subsection{The approximation problem}\label{1.1}  Approximating
 'bad' functions by 'good' ones is one of classical topics in 
 Analysis. The 
 qualitative direction
 has started with the Weierstrass Theorem on the possibility of 
 polynomial 
 approximation
 of continuous functions. An important further development here concerns 
 approximating  
 continuous functions by solutions of differential equations. A 
 fundamental result 
 for rather general differential equations (possessing a kind of 
 unique 
 continuation property) was obtained by F.E. Browder, \cite{BR1}, 
 \cite{BR2}.

 The studies in the quantitative direction began later.
  Generally speaking, quantitative approximation results can be 
  expected to have 
  the following common structure:
 \begin{enumerate}
   \item A class  $\Fs$ of functions to be approximated is 
       described; 
   \item A class $\Gs$ of functions used for approximation is 
       proposed;
   \item The result: a quantitative relation between the rate of 
       approximation 
       and the properties of the 
   approximating function.
 \end{enumerate}
 For example, the order of the error in the approximation of a 
 continuous function 
 by polynomials
  of a given degree is determined by the smoothness of this 
  function, understood in 
  a proper sense.
  
  When considering approximation by solutions of elliptic 
  equations, it is 
  reasonable to consider as $\Fs$, a class  of functions defined 
  on a nowhere dense 
  set $\Kb$. In fact, if, on the opposite, $\Kb$ possesses 
  interior points, it is 
  only solutions of the  equation that can be approximated by 
  solutions.  
   So, we are interested in approximating a given continuous 
   function $f$ defined 
  on a  nowhere dense compact
set $\Kb\subset \R^\Nb$ by solutions of a second order elliptic 
equation. When the 
approximating functions are harmonic,
and the set $\Kb$ is nice, say, a Lipschitz surface, there are many 
results in this 
direction, see, e.g. \cite{Bliedtner}, \cite{AlSh1}, 
\cite{Andrievskii}, \cite{Gardiner Book}, \cite{Gardiner}, 
\cite{Gardiner Goldstein}, \cite{Hausmann}, \cite{Khav} and many 
more.

 When the conditions on $\Kb$ are less restrictive, one can cite 
 \cite{AlSh2}, \cite{Pavlov}, \cite{RZ25}. Here, 
 one needs to decide,  which terms should be used 
to describe properties of the function $f$ in order  to determine 
the rate of 
 approximation.
  If we only know that the given function $f$ is continuous on a compact set 
  $\Kb$, then  the 
  quality of this continuity, and consequently,
   the quality of approximation, can be described by the modulus 
   of continuity of 
   $f$. In this direction,  in the paper \cite{RZ25}, the authors 
   considered the 
   problem on  approximating  a continuous function $f$ on 
 $\Kb\subset\R^{\Nb}$, possessing the continuity
  modulus $\om(\de)$, by solutions of a second order 
 elliptic equation $\Lc u=0$ (\emph{$\Lc$-harmonic functions}).
  It was established there that if the set 
  $\Kb$ is 
  Ahlfors-David $\Nb$-$2$-regular (which means, almost exactly 
  speaking, that
  it has one and the same Hausdorff dimension $\Nb$-$2$ in any 
  neighborhood of any 
  of its points), then the function $f$ can be,
  for any $\de>0$, approximated in $C(\Kb)$ by a function $v_\de$, 
  so that 
  $|f(x)-v_{\de}(x)|\le c \om(\de)$ for all $x\in\Kb$,
  the function $v_{\de}$ is $\Lc$-harmonic in  a 
  $\de$-neighbourhood  $\Kb_{\de}$ 
  of the set $\Kb$, moreover,
   the quality of this function $v_\de$ is controlled by $\de$, 
   namely, $|\nabla 
   v_\de(x)|\le C \frac{\om(\de)}{\de}$ in $\Kb_{\de}$. This matches the 
   general 
   principle: the smaller $\de$,  the better is the aproximation, i.e.,  the smaller, 
   is the 
   approximation error, but the worse is the approximating
    function $v_\de$: it is $\Lc$-harmonic on a smaller set, and 
    its gradient may 
    grow with $\de$ decreasing. Moreover, a converse result was 
    established: if a 
    continuous function $f$  on $\Kb$ can be approximated in the 
    above sense, with 
    some function $\om(\de)$, by solutions of a second order 
    elliptic differential 
    equation, then $f$ possesses the continuity modulus majorated 
    by $\om(\de)$.
    
    It is natural to expect  that if we wish to have a better 
    approximation (the one better than with 
    $O(\om(\de))$ error), with the same quality of
    the approximating function, we should suppose some better 
    properties of the given
    function $f$. If the set $\Kb$ were a smooth
    surface (of codimension 2), such 'better' properties would 
    naturally involve a 
    higher classical  smoothness of $f$. However, if we only know 
    that the set 
    $\Kb$ is Ahlfors-David $\Nb$-$2$-regular, some other terms 
    should be used.

    In the literature, there exist methods of defining spaces of 
    'nice' functions on 
    arbitrary compacts. One of them is based upon
    describing
   classes of functions via their local approximations by 
    polynomials or other
     sufficiently regular functions, see \cite{Brudny}, 
     \cite{Brudny2}, 
     \cite{Shvartsman}, and many sources afterwards.
  
    So, the expected approximation results should sound like 'if a 
    function admits 
    local approximation of a certain kind, it admits
     the corresponding quality of global approximation' by 
     $\Lc$-harmonic 
     functions. 
     
     This is, in fact, the contents of the present paper. Namely, 
     in our main 
     result,
   if $f$ is a continuous function on $\Kb$,  which can, for any 
   $\de>0$, be 
   locally, in a $\de$-neighborhood of any point $x\in\Kb$, 
   approximated by a function $\F_{x,\de}(y)$ which is a solution 
   of the 
   second-order elliptic  equation 
$\Lc(y,\partial_y)u(y)=0$ in a $2\de$-neighbourhood of $x$, with 
error 
$O(\de^\rb\om(\de)), $ $\rb\ge1$  (with some natural compatibility
 conditions concerning the functions $\F_{x,\de}$ for different values of
 $\de$ and different close-lying points
 $x$),
    then $f$ can be approximated on the whole $\Kb$, with error of 
    the same order, by 
    a solution $v_\de$ of the same equation in the  
    $\de$-neighborhood of $\Kb$.
Note that the above compatibility conditions, mentioned in 
parentheses, are 
unavoidable: they are proved to be necessary 
for the existence of the  global approximation.

When comparing these results with our previous paper \cite{RZ25}, 
where  we 
established this  kind of properties for $\rb=0$,
 one can notice that an additional restricting  condition appears: 
 the locally approximating
functions $\F_{x,\de}(y)$ are required here to be solutions of the  
elliptic 
equation, while in 
\cite{RZ25} no such restriction has been imposed.
This restriction is, unfortunately, unavoidable.  An example we 
present in the 
paper  demonstrates a function which admits a nice polynomial 
local approximation but 
does not admit a global  approximation by harmonic functions.
 This effect is caused by  a visible wildness of the set $\Kb$ in 
 our example: it is easy to show 
 that for a nicer $\Kb$, e.g., for a Lipschitz
 surface of codimension 2, such counter-examples are impossible 
 and a local 
 approximation by smooth functions is sufficient (and, of course, 
 necessary)
  for existence of a global approximation by    $\Lc$-harmonic 
  functions.

 The elliptic  differential operation $\Lc(x,\partial_x)$ is 
 supposed to have 
 coefficients of certain finite smoothness,
  $C^m(\Om)$. The main approximation result, Theorem 1.2, is proved for $m=3$. Under 
  additional  
  smoothness conditions, the main result can be somewhat
  strengthened: not only the approximating functions $v_\de$ 
  converge on $\Kb$ to 
  the initial function $f$, but their derivatives $\partial^\a 
  v_{\de}$ (up to 
  some order, depending on the smoothness of coefficients of 
  $\Lc$) converge on 
  $\Kb$ to some functions $f_{(\a)}$ which can be understood as 
  generalized 
  derivatives of the given function $f$. The greater $m$, to the higher order these surrogate 
  derivatives of $f$ can be defined, see Theorem 1.3.

 \subsection{The main results.1}\label{Sect.Main. Result1}
 We present here the exact formulation of our main approximation 
 result.  It is the 
 following.
  Let $\om(t),$ $t>0$, be a continuity modulus satisfying the 
  condition
  \begin{equation*}
 % \begin{equation}\label{Cont.mod}
    \int_0^{\ta}\frac{\om(t)}{t}dt+\ta\int_{\ta}^\infty\frac{\om(t)}{t^2}dt \le c\om(\ta), \, 0<\ta<\infty.  \end{equation*}
Let, further, $\Kb$ be a compact set in $\R^\Nb$, 
$\Nb$-$2$-Ahlfors-David regular (see, 
e.g., \cite{David}).
Let $\Om\supset\Kb$ be a bounded open connected set, where a 
formally self-adjoint 
second-order elliptic operator 
\begin{equation*}
\Lc u (x)=-\sum_{\jt,\jt'} \partial_\jt(a_{\jt\jt'}(x)\partial_{\jt'}u(x))\equiv 
-\nabla\cdot(\ab(x)\nabla u(x)),
\end{equation*}
with $C^m$-coefficients $a_{\jt\jt'}$, $m\ge 3$,
is defined.

With the continuity modulus $\om$ fixed, for an integer $\rb\ge0$, 
the local 
$\Lc$-H{\"o}lder class $\Hc_{\Lc}^{\rb+\om}(\Kb)$ is defined in 
the 
following way.
\begin{defin}\label{defin.class}
  The continuous function $f(x)$, $x\in\Kb,$ is said to belong to 
$\Hc_{\Lc}^{\rb+\om}(\Kb)$, if there exist constants 
$\cb_1=\cb_1(f)$, 
$\cb_2=\cb_2(f)$, such that for any $x\in\Kb$ and any $\de,\, 
0<\de\le 
2\diam(\Kb)$, there exists 
a function $\F_{x,\de}(y)$ defined in the ball $B_\de(x)$  such 
that
\begin{equation*}%\label{LF}
  \Lc_y\F_{x,\de}(y)=0, \, y\in B_\de(x),
\end{equation*}
\begin{equation*}%\label{F1}
  |f(y)-\F_{x,\de}(y)|\le \cb_1\de^\rb\om(\de), \, y\in 
  B_\de(x)\cap\Kb.
\end{equation*}
For close-lying points $x_1$, $x_2$, the approximating functions 
 should be consistent in the following sense: for some constants 
 $\g_1,\g_2,$ 
 $\frac18\le \g_1\le 1\le \g_2\le 8,$
 if $\g_1\de_1\le \de_2\le \g_2\de_1$,  given any points 
 $x_1,x_2\in \Kb$, such that the 
 balls $B_{\de_1}(x_1), B_{\de_2}(x_2)$ are not disjoint, the 
 inequality
\begin{equation}\label{F2}
 |\F_{x_1,\de_1}(y)-\F_{x_2,\de_2}(y)|\le 
 \cb_2\de_1^\rb\om(\de_1). 
\end{equation}
must hold for all $y\in B_{\de_1}(x_1)\cap B_{\de_2}(x_2)$,
\end{defin}

We recall the definition of Ahlfors-David regularity. The compact 
set $\Kb$ is called 'AD regular' of dimension $\vk$ if for some 
constants
$c', c'',$ $0<c'<c''<\infty$, for any point $x\in \Kb$ and any 
$r\le\diam(\Kb)$,
$c'r^{\vk}\le\m_\vk(B_r(x))\le c''r^{\vk},$ where $\m_{\vk}$ is the Hausdorff measure of dimension $\vk.$

By $\Kb_\de$ we denote the $\de$-neighborhood of $\Kb.$
Our first main result is the following.
\begin{thm}\label{MainTheorem}
 Let $\Kb$ be $\Nb$-$2$-AD regular. Suppose that the coefficients of the operator $\Lc$ belong to $C^3$. Then   function 
 $f$ defined on $\Kb$ 
 belongs to the class $\Hc_{\Lc}^{\rb+\om}(\Kb)$ if 
and only if
for any $\de< \frac14\diam(\Kb)$, there exists an approximating 
function
 $v_\de(x), \, x\in\Kb_\de,$ such that, with some constant
$\mathbbm{c}>0$,
\begin{gather}\label{appr.function}
  \Lc_y v_\de(y)=0, \, y\in \Kb_{\de};\\\nonumber
  |v_\de(x)-f(x)|\le \mathbbm{c} \de^\rb\om(\de), \, x\in 
  \Kb;\\\nonumber
  |v_\de(y)-v_{\de/2}(y)|\le \mathbbm{c}\de^\rb\om(\de), \, 
  y\in\Kb_{\frac{\de}{2}}.
\end{gather}
\end{thm}

\subsection{The ideas of the proof} The proof of the main 
theorem is fairly 
technical, therefore we consider it reasonable to explain here its 
structure.

 Given 
a function $f\in \Hc_{\Lc}^{\rb+\om}(\Kb)$ on the compact set 
$\Kb$, we construct its 
special extension $f_0$ to a fixed neighborhood $\Om$ of $\Kb$ 
(the particular form of this
 neighbourhood is not essential, and we suppose further on that it is the 
 unit ball containing the set $\Kb$ which is contained in the concentric ball with radius $\frac13$). For this function 
$f_0$,  using the Green function $ G^{\circ}(x,y)$  of the operator $\Lc$ 
in $\Om,$ the 
integral representation is established:
\begin{equation}\label{intro.repres}
  f_0(x)=\int_{\Om}\Lc f_0(y)G^{\circ}(x,y)dy, \, x\in \Om.
\end{equation}
Although this representation looks quite usual if $f_0$ is 
sufficiently smooth, 
this is not the case for \emph{our} function $f_0$  for which the 
derivatives may 
behave badly when approaching $\Kb$. Therefore, to justify 
\eqref{intro.repres}, we need a detailed control of the behavior 
of  $\Lc f_0(x)$ 
and of  derivatives of $f_0(x)$    near $\Kb$. Obtaining  this 
control requires 
complicated estimates of the Green function $G_{x,r}(x,y)$ for $\Lc$ 
in balls 
$B_r(x)$ centered at $x$, together with their derivatives, up to 
the third 
order, in the variables $x,y$, as well as in the additional variable $\vs$ on which the operator $\Lc$ depends as a parameter. 
Under the condition 
of a sufficient smoothness of coefficients of the operation $\Lc,$ 
we derive some of these estimates directly,   using Schauder-type 
 approach, and borrow the other ones 
 from the results by Ju.~Krasovskii 
 \cite{Kras.Sing}, and M.~Gr{\"u}ter--K.-O.~Widman \cite{Widman}. 
Finally, having established    the representation 
\eqref{intro.repres}, we define 
the approximation function $v_\de(x)$, looked for, by the integral
\begin{equation*}
  v_{\delta}(x)=\int_{\Om\setminus \Kb_{\de}}\Lc f_0(y)G^{\circ}(x,y)dy,
\end{equation*}
with addition of a collection of  several  compensatory $\Lc$-harmonic terms of a simpler nature, see \eqref{35}.
 The fact that $v_{\de}$ is 
$\Lc$-harmonic in $\Kb_{\de}$ is obvious, it follows from the 
definition of the 
Green function $G^{\circ}(x,y)$, while the estimates producing  the 
quality   of the  
approximation follow from   our estimates for the function 
$f_0(x)$ and its 
derivatives.

\subsection{The main result. 2}\label{Sect.Main. Result2}
The second theorem describes the properties of the approximating 
functions $v_\de:$ 
their derivatives, up to a prescribed order 
$k\le\rb+1$ can be controlled. Moreover, we 
can define in a consistent way the generalized derivatives $f_{(\a)}$ of the 
initial function 
$f$ on $\Kb$,
so that the derivatives of $v_{\de}$ approximate these derivatives 
of $f$. This 
property requires a certain additional smoothness of coefficients 
of the operator 
$\Lc$.  
\begin{thm}\label{Thm.quality} Suppose that $\rb\ge 1$ and the 
coefficients $\ab(x)=(a_{\jt\jt'}(x))_{\jt,\jt'\le \Nb}$ belong to 
$C^{k_0+3}(\Om)$ for a certain $k_0\le\rb$.  Let the 
function $f$, defined  on 
the compact set $\Kb$, belong to the class $H_{\Lc}^{\rb+\om}(\Kb)$ and 
$v_\de$ be its 
approximation, as in \eqref{appr.function}. Then  derivatives of 
$v_\de$ satisfy
\begin{equation}\label{appr.deriv}
  \|\nabla^{k_0+1}v_\de\|_{\Kb_{\de/2}}\le c\frac{\om(\de)}{\de}.
\end{equation}
moreover, surrogate derivatives $f_{(\a)}(x)$ can be defined, so that
\begin{equation}\label{appr.deriv.2}
  |f_{(\a)}(x)-\partial^\a v_{\de}(x)|\le C 
  \de^{\rb-|\a|}\om(\de),\, x\in\Kb,  \, 
  1\le |\a|\le k_0.
\end{equation}
\end{thm}
\subsection{Structure of the paper} We start in Sect. \ref{Sect.Prep} by 
presenting  general
material concerning   certain geometry considerations, and formulate estimates 
of important integrals used in further analysis and of derivatives of the Green function, 
including the results of \cite{Kras.Sing} and \cite{Widman},
 In Sect. 3, we introduce the 
 averaging kernel $K(x,y)$ and prove 
 estimates of its derivatives. This is the most technical part of 
 the paper. Next, in Sect. 4, we construct the
  extension function $f_0$, derive  its important  properties and 
  prove its integral representation, which
  results in presenting the required approximation of the given function 
   $f(x)\in H_{\Lc}^{\rb+\om}(\Kb)$, thus proving Theorem 
   \ref{MainTheorem}. In Sect 5, we 
   discuss generalized derivatives of the function $f$, and prove 
   Theorem \ref{Thm.quality}. Then, in Sect.6, we present the 
   example
   showing that for a wild set $\Kb$, the condition on local 
   approximation cannot, generally, be relaxed.
   
   Proofs of our estimates for derivatives of the Green function and of important integral inequalities
     are placed 
    in the Appendix.

\subsection{Conventions}In the course of the paper, we denote by 
the same symbol 
$c$ or $C$ various constants whose particular value is of no 
importance, as long as 
this does not cause confusion; sometimes, subscripts or superscripts 
are used in 
order to distinguish between such constants in the same formula.  
More important 
constants may be highlighted by a different font. By 
$f'_x=\partial_x f=\nabla_x f$ we denote the $x$-gradient of a 
function $f$;
for a vector function $F$, $\nabla_x F$ stands for the Jacobi 
matrix of $F$. The symbol $| \cdot |$ denotes the Euclidean norm of 
the vector involved, $\mathbf{E}$ denotes the unit matrix.

\section{Some preparatory facts}\label{Sect.Prep}

\subsection{Geometry considerations}\label{SEct.geometry}
Let $b_{\Nb}$ be the constant in the covering property of 
Ahlfors-David--regular sets of 
dimension $\Nb$-$2$,
 see \cite{Mattila}, Lemma 2.1, and \cite{RZ25}, Corollary 2.2 
 there,
  namely,
\begin{property}\label{MattilaProperty}
For any $\de<\diam(\Kb),$ there exists a finite cover $\U(\de)$ of 
$\Kb$ by open balls 
$B_\de(x_{\am})$ of radius $\de$,
\begin{equation*}
  \Kb\subset\bigcup_{x_{\am}\in\Kb}B_\de(x_{\am})\equiv 
  \Kb_{(\de)},
\end{equation*}
such that for any $r\in[\de,\diam(\Kb)]$ and any point 
$\x_0\in\Kb$, the quantity of 
points
 $x_{\am}$ in the ball $B_{r}(\x_0)$ is not greater than 
 $b_\Nb\left(\frac{r}{\de}\right)^{\Nb-2}$.
\end{property}

Our aim at this moment is to associate, with each ball 
$B_\de(x_{\am})$ of the above 
cover, some new ball with radius $2\de$, whose center  is  on the order 
$\de$ distance from 
$x_{\am}$ and which is separated from $\Kb_{(\de)}$, again, by an order 
$\de$ distance.

We denote by $\si_{\Nb}$ the area of the unit sphere in 
$\R^{\Nb}.$ Next, we 
introduce some coefficient $A=A_{\Nb}$,
 whose value will be determined  later on in a special way.
Thus, for any $x$, for the sphere $S_{\de A}(x),$ its 
$\Nb$-$1$-dimensional 
surface
 measure equals $\si_{\Nb} A^{\Nb-1}\de^{\Nb-1}$.

 We take some point $x_{{\am}_0}$ among  centers  of the balls in 
 the cover $\U(\de)$ in 
 Property \ref{MattilaProperty} and denote temporarily by 
 $\x_0\equiv \x_{{\am}_0}$ (one of) the 
 point(s) in $\Kb$ closest to $x_{{\am}_0}$ (it may happen that these 
 points coincide), so
 $|\x_0-x_{{\am}_0}|\le\de.$ 
  For the  ball $\Bb:=B_{\de (A+4)}(\x_0),$  there exist no more 
  than 
 $b_{\Nb}(A+4)^{\Nb-2}$  balls $B_\de(x_{\am})$ in the above  
 cover $\U(\de)$, whose  
 centers $x_{\am}$ lie  in $\Bb$. 
We place on the sphere $\Sbb:=S_{\de A}(\x_0)$,  in an arbitrary 
way, a collection 
of $m\le b_{\Nb}(A+4)^{\Nb-2}$ 
points $\z_k,$ and evaluate the area on the part of the sphere 
$\Sbb$
 covered by the union of   balls with radii $6\de$, centered at 
 these points $\z_k$;
 we denote this area by $\si(\x_0,\de)$.  This area is no greater 
 than the sum of 
 areas of  spherical
  caps upon $\Sbb$, covered by single balls, therefore,
\begin{equation}\label{29}
  \si(\x_0,\de)\le 6^{\Nb-1}\si_{\Nb} 
  b_{\Nb}(A+4)^{\Nb-2}\de^{\Nb-1}.
\end{equation}
We denote by $\tilde{A}_{\Nb}$ the largest positive root of the 
equation
\begin{equation*}
  6^{\Nb-1}b_{\Nb}(\tilde{A}_{\Nb}+4)^{\Nb-2}=\frac12 
  \tilde{A}_{\Nb}^{\Nb-1}
\end{equation*}
and set 
\begin{equation}\label{30}
  A=\max(\tilde{A}_{\Nb}, 13).
\end{equation}
It follows from \eqref{29}, \eqref{30} that whatever  points  
$\z_k$, no more than  $ 
b_\Nb(A+4)^{N-2}$ of them,  we place on the sphere $\Sbb$,   at least a half 
of the area of this sphere is \emph{not} covered by the concentric 
   balls $B_{6\de}(\z_k)$. We denote this, non-covered, part of 
   the sphere  by 
   $\Yc\equiv \Yc_{x_0,\de}(\{\z_k\})$, so,
 
 \begin{equation*}%\label{31}
 \meas_{\Nb-1}(\Yc_{x_0,\de}(\{\z_k\}))\ge \frac12 
 \meas_{\Nb-1}\Sbb.
 \end{equation*}

 Next we consider the following geometrical construction. We 
 choose the above 
 points $\z_k$ in a special way.  Denote by $\Pc$ the \emph{closed} 
 spherical annulus 
 $\Pc=\overline{B_{A_\Nb+4\de}(\x_0)}\setminus 
 {B_{A_\Nb-4\de}(\x_0)}$ and 
 consider \emph{ only} those points $x_{\am}$ which lie in $\Pc$. 
 Suppose that a certain 
 point $x_{\am}$ lies on the sphere $\Sbb$. Then we set 
 $\z_{\am}=x_{\am}$.  If 
 $x_{\am}$ \emph{does not} lie on this sphere, we consider the 
 straight ray, which we 
 denote  $\overrightarrow{[\x_0 ,x_{\am}]}$,  starting at $\x_0$ and 
 passing through 
 $x_{\am}$, and accept as $\z_{\am}$   the point where this ray 
 hits the sphere 
 $\Sbb.$ As explained above, the set $\Yc$ is non-empty. We take 
 an arbitrary point 
 $V$ in this set.
 
 Our construction started with choosing a  point $\x_0\equiv 
 x_{\am_0}$ in Property 
 \ref{MattilaProperty}. To reflect it, we mark the point $V$, just 
 defined, as 
 $V_{\am_0}$, thus keeping $\am_0$ fixed.
  
 We are going to estimate from below the distance between 
 $x_{\am}$ and $V_{\am_0}$.
 For a point $x_{\am}$ in $\Pc$ and $\z_{\am}\in \Sbb$, we have 
 $|V_{\am_0}-\z_{\am}|\ge 6\de.$ 
 The point $x_{\am}$ lies on the ray 
 $\overrightarrow{[\x_0,\z_{\am}]}$, therefore, 
 $|x_{\am}-V_{\am_0}|$ is not less than the length of the perpendicular 
 dropped from $V_{\am_0}$ 
 onto $\overrightarrow{[\x_0,\z_{\am}]}$. Since 
 $|\x_0-\z_{\am}|=|\x_0-V_{\am_0}|\ge 13\de,$ the 
 length of this perpendicular is not less than $4\de$. Therefore, 
 for $x_{\am}$, we 
 have
 \begin{equation}\label{A32}
   B_{2\de}(x_{\am})\cap B_{2\de}(\x_0)=\varnothing.
 \end{equation}
 If, on the opposite, the point $x_{\am}$ does not lie in $\Pc$ then 
 \eqref{A32} 
 obviously holds.
 
 In this way, with each starting point $ x_{{\am}_0}$, we 
 associate the point 
 $V_{{\am}_0}$ such that $B_{2\de}(x_{\am})\cap 
 B_{2\de}(V_{{\am}_0})=\varnothing$ 
 for all ${\am}\ne {\am}_0$. Additionally, 
 
 \begin{equation*}
   |V_{{\am}_0}-x_{{\am}_0}|\le 
   |V_{{\am}_0}-\x_0|+|\x_0-x_{{\am}_0}|\le 14 \de,
 \end{equation*}
 and thus the point $V_{{\am}_0}$ is separated from $\Kb$

Finally, we introduce the notation
\begin{equation*}
  \Kb'_\de=\bigcup_{x_{\am}}B_{2\de}(x_{\am});
\end{equation*}
it follows that $\Kb_\de\subset\Kb'_\de$

\subsection{Estimates of some  integrals}\label{Sect.Integrals}
  In the study of approximations, we will need  estimates for some 
  integrals 
  involving the distance to the set $\Kb$.  Here we give the formulations; proofs are placed 
  in the Appendix. 

Let $\Kb\subset\R^\Nb$ be a compact set, $\Nb$-$2$-regular, and 
  let the point $x$ 
  lie outside $\Kb$, $d(x)\equiv\dist(x,\Kb)\le\de.$
  \begin{lem}\label{lem2}
\begin{equation}\label{21}
\int_{B_{2\de}(x)}\frac{d(y)^{\rb-2}\om(d(y))}{|x-y|^{\Nb-2}} dy
\le C \de^{\rb}\om(\de).
\end{equation}
  \end{lem}

  Lemma \ref{lem2}, has a useful corollary.
  \begin{cor}\label{Cor.L2}
    Under the same conditions imposed on $x$,
    
    \begin{equation}\label{22}
      \int_{B_{2\de}(x)} d(y)^{\rb-2}\om(d(y))dy\le c 
      \de^{\rb+\Nb-2}\om(\de).
    \end{equation}
  \end{cor}

   Another important property concerns the integral  
   \begin{equation}\label{int.4}     
   I_k(x)=\int_{B_{c_0\de}(x)}\frac{d(y)^{k-2}\om(d(y))}{|x-y|^{\Nb-2+k}}dy, 
   \, 
   x\in\Kb, c_0\ge2.
   \end{equation}
   \begin{lem}\label{lemA1} For $k>0,$ the inequality holds
   \begin{equation}\label{lem.int.4}
    I_k(x)\le c\om(\de(x)).
   \end{equation}
   \end{lem}
   
   There is a useful corollary: 
  \begin{cor}\label{Cor2} For $1\le k<\rb$,
   \begin{equation}\label{A6a}
  J_k= 
  \int_{B_{c_0\de}(x)}\frac{d(y)^{\rb-2}\om(d(y))}{|y-x|^{\Nb-2+k}}dy\le 
  c 
  \de^{\rb-k}\om(\de).
   \end{equation}
   \end{cor}
    
   In the next lemma we estimate the integral over the complement 
   $\Ct    B_{c_0\de}(x)$ of the ball.
\begin{lem}\label{Lemma2} For $k>0$, we have
\begin{equation}\label{A7}
\int_{\Ct} 
B_{c_0\de(x)}\frac{d(y)^{k-2}\om(d(y))}{|y-x|^{\Nb-1+k}}dy\le 
c\frac{\om(\de)}{\de}.
\end{equation}
\end{lem}

\begin{lem}\label{Lem3a}Let $\rb\ge1$, $x_0\in\Kb,$ $x\in \mathrm{C}\Kb$ and $|x-x_0|\le 
\de/2$. Then
\begin{equation*}%\label{eq.Lem3A}
  \int_{B_{\de}(x)}d(y)^{\rb-2}\om(d(y))|x-y|^{2-\Nb}dy\le C \de^r\om(\de)
\end{equation*}
\begin{equation*}%\label{eq.Lem3B}
 \int_{B_{\de}(x_0)}d(y)^{\rb-2}\om(d(y))|x_0-y|^{2-\Nb}dy\le C\de^\rb \om(\de).
\end{equation*}
\end{lem}

Finally, we need the following estimate.

\begin{lem}\label{Lem4}Under conditions of Lemma \ref{Lem3a},
\begin{equation}\label{eq.Lem4}
 \Jc(x)= \int_{|x-y|>\de}d(y)^{\rb-2}\om(d(y))|x-y|^{1-\Nb}dy\le C\de^{-1}\om(\de).
\end{equation}
\end{lem}

\subsection{Estimates of derivatives of the Green function}\label{Sect2.3}
Here we present some results on the Green function and its derivatives.
  
 Let $\Lc=\Lc(x,\partial_x)$ be a uniformly elliptic order $2$
operator in a ball  $B_R$ of radius $R\le1$, with Dirichlet boundary conditions and $G(x,y)\equiv G^{R}(x,y)$ be 
the Green function for $\Lc$ in $B_R$. 

Estimates of the first type concern the unit ball,  $ R=1,$ $B=B_1$, $G=G^{\circ},$ and we need them for 
\emph{\emph{all}} points $(x,y)\in(\bar{B}\times  \bar{B})$, $x\ne y$. Our  interest lies in 
estimating the singularity of the derivatives of $G(x,y)$ as $x$ is close to $y$.
Such estimates for $G(x,y)$ and  derivatives $G_x(x,y)$, $G_y(x,y)$, 
$G_{xy}(x,y)$ were established in \cite{Widman}, under rather weak restrictions imposed on the 
coefficients, namely, if $\ab(x)\in L_{\infty}(B)$ and is Dini continuous. Then
\begin{equation}\label{WidmanEst}
  |\partial^\a_x\partial^\be_y G(x,y)|\le C|x-y|^{-\Nb + 2-|\a|-|\be|},
\end{equation}
for $|\a|, |\be|\le 1$, with constant $C$ depending on the norms of coefficients in the Dini 
class and on the ellipticity constant.

Estimates of derivatives of higher order require more regularity of coefficients. We need them 
only for derivatives $\partial_x^{\a}\partial_y^{\be}G(x,y),$ $|\be|\le 1$.  We cite here the result by Yu.Krasovskii, tailored 
for our particular case, \cite{Kras.Sing}, see Theorem 3.3. and its corollary.
\begin{thm}\label{Kras.thm}Let the coefficients $\ab(x)$ belong to $C^m(\bar{B}),$ $m\ge 3$. 
Then
\begin{equation}\label{Kras.est}
  |\partial_x^{\a} \partial_x^{\be}G(x,y)|\le C |x-y|^{-\Nb+2-|\a|-|\be|}, \, |\a|<m-1,
\end{equation}
with constant $C$ depending on $C^m$-norm of coefficients and the ellipticity constant.
\end{thm} 
We will also need  such estimates of derivatives for the Green function $G^{R}(x,y)$ in a ball $B_{R}$ with radius $R<1$.
\begin{cor}\label{cor.Kras} Under the conditions of Theorem \ref{Kras.thm}, for a ball $B_{R}$, $R<1$ estimate \eqref{Kras.est} holds for the Green function $G^R(x,y)$ in the ball $B_R,$ with constant with constant $C$ depending on $C^m$-norm of coefficients and the ellipticity constant, but not depending on $R$.
\end{cor}
\begin{proof} Having an operator $\Lc$ in the ball $B_R$, say, 
with center at the origin,  we make a dilation $\xb=R^{-1}x$ to 
the unit ball $B_1$. Under this dilation, the $C^0$-norm of a 
derivative of order $\a$ in $x$ or $y$ gains the factor $R^{|\a|}<1.$ 
Thus, the 
$C^{m}$ norms of the coefficients do not grow. As  follows from 
the chain rule, the Green function in $x$ variables is $R^{2-\Nb}$ 
times the transformed Green function of the dilated operator in 
$\xb$ variables,
\begin{equation*}
  G^R(x,y)=R^{2-\Nb}G^{\circ}(\xb,\yb).
\end{equation*}
and the estimate for the Green function   follows from the estimate in the unit ball. The same dilation takes care of derivatives of the Green function.
\end{proof}

The second type of results concerns the Green function in a ball $B_R$ with radius $R\le 1$ 
for the points $x,y$ \emph{well separated}, namely, the point $x$ lies in a small neighborhood of the 
centerpoint $\mathbb{0}$ of the ball, while $y$ lies in a neighborhood of some  boundary point 
$y^{\circ}$. The operator $\Lc$ depends on an additional parameter $\vs$ in a neighborhood of the point
$\vs_0=\mathbb{0}\in \R^\Nb$, and we need estimates for the derivatives 
   
   \begin{equation}\label{A.unit ball}
     |G_x|, \, |G_y|,\, |G_{xy}|,\,  |G_{x\vs}|,\, |G_{xy\vs}|,\, 
     |G_{xx\vs}|, |G_{yy\vs}|, \,|G_{x\vs\vs}|, \, |G_{xxy}|\le C,
   \end{equation}
   for such $x,y$, uniform in these variables and for $\vs=0$.

Again, we start with the unit ball , $R=1$, and the Green 
function,  denoted $G^{\circ}(x,y;\vs).$ 

 \begin{thm}\label{Th.A1}Suppose that the operator $\Lc(\vs)=-\nabla (\ab(x,\vs)\nabla)$ is 
 uniformly elliptic and its coefficients satisfy
 $\partial_x^{\a}\partial_\vs^{\g}\ab(x,\vs)\le c_0,$ for $|\a|+|\g|\le 3.$  Then for  $x$ near 
 $\mathbb{0}$, $y$ near $y^{\circ}\in \partial B_1$, $|x-y|>\frac12$, the derivatives of 
 the Green function $G^{\circ}(x,y,\vs)$, satisfy \eqref{A.unit ball},
 with constant $C$ depending on the norm of the coefficients in the above spaces and on the 
 ellipticity constant.
 \end{thm}
 The proof is presented in Appendix B.
 
 Again, a  simple consequence of Theorem \ref{Th.A1} is the following. 
 \begin{cor}\label{Cor.A1} Under the conditions of Theorem 
 \ref{Th.A1}, let $B_R\subset B_1$ be a ball with radius $R<1$, 
Then for the  Green function $G^R$ for the 
 operator $\Lc$ in $B_R$, the estimates hold. 

  \begin{gather}\label{A.R ball}
      |G^R_\vs(x,y;\vs)|, |G^R_{\vs\vs}(x,y;\vs)|\, \le 
     CR^{2-\Nb};\\\nonumber 
     |G^R_x(x,y;\vs)|, \, |G^R_y(x,y;\vs)|, \,|G^R_{x\vs}(x,y;\vs)|, \, |G^R_{x\vs\vs}(x,y;\vs)|\le 
     CR^{1-\Nb},\\\nonumber
     |G^R_{xy}(x,y;\vs)|,\, |G^R_{xy\vs}(x,y;\vs)|,\, |G^R_{xx\vs}(x,y;\vs)|\le C R^{-\Nb},\\\nonumber
      |G^R_{xyy}(x,y;\vs)|\le C R^{-1-\Nb},
   \end{gather}
   with constant $C$ determined by $C^3(B_R\times B_{\e})$- 
   norms of coefficients of the operator $\Lc$ and its ellipticity constant.
 \end{cor}
\begin{proof}The proof repeats  the above dilation reasoning.
\end{proof}

\section{The generating kernel and the 
extension}\label{Sect.K(x,y)}
\subsection{The extension operator}\label{Extension}
We are going to describe here the smooth extension to $\R^\Nb$ of a 
given 
function
 $f\in \Hc_{\Lc}^{\rb+\om}(\Kb)$  with control of 
 derivatives.
 
 Let $\Qc=\{Q\}$ be the Whitney decomposition into open cubes of 
 the set 
 $\Ct\Kb=\R^{\Nb}\setminus\Kb$. Recall 
 that this means that $\bigcup_{Q\in\Qc}\overline{Q}=\Ct\Kb$ and 
 different cubes in 
 $\Qc$  are disjoint. Moreover,
 if $\ab_Q$ is the center of the cube $Q\in\Qc$, 
 $\de(Q)=\dist(\ab_Q,\Kb)$, then 
 $\frac18\de(Q)\le \diam Q\le\frac14\de(Q)$.
 For $Q\in\Qc$, we denote by $x_Q\in\Kb$ (one of) the closest  to 
 $\ab_Q$
 point(s) in $\Kb$, $d(\ab_Q)=|x_Q-\ab_{Q}|,$ and for $y\in Q,$ $\de(Q)\le 2 \diam\Kb$, we define
\begin{equation}\label{f-tilde}
  \fT_0(y)=		\F_{x_Q,2\de(Q)}(y) 
\end{equation}
where $\F_{x_Q,2\de(Q)}(y)$ are functions involved in the 
definition of the class 
$\Hc_{\Lc}^{\rb+\om}(\Kb)$, see Definition \ref{defin.class}. If a point $y$ does not belong 
to any of $Q\in\Qb$ or it belongs to a far-lying cube, $\de(Q)> 2 \diam\Kb,$ we set $ 
\fT_0(y)=0.$
  The function $\fT_0,$ thus defined, is piecewise $\Lc$-harmonic, 
  however it may 
  be discontinuous along the boundaries
  of cubes in $\Qc$
  in an uncontrollable manner.  We use now the averaging kernel 
  $K(x,y)$, to be  
  constructed later in this section,
  and set
  \begin{equation}\label{function f}
   f_0(x)=\int_{\R^{\Nb}} \fT_0(y) K(x,y) dy.
  \end{equation}
We will show that this function is a continuous extension of $f$ 
to a neighborhood 
of $K,$
 with controlled behavior of derivatives when approaching $\Kb$; the function $f_0(x)$ will 
 later serve  for 
 constructing the required 
 approximation.
\subsection{Construction of the kernel 
$K$}\label{subs.construction K}  Our 
reasoning will be constructive. The first 
step will be  describing a proper averaging kernel. We denote  by 
$\dn(x)$ 
the \emph{regularized} distance from the point $x\in \Ct\Kb$ 
to $\Kb$, namely $\dn(x)\in C^3(\Ct\Kb)$, $c\,\dist(x,\Kb)\le 
\dn(x)\le 
c'\,\dist(x,\Kb),\, c'<\frac14,$ $|\grad^k\dn(x)|\le c 
\dn(x)^{1-k}, \, k=1,2,3.$ 
Let $\hb(t)$ be a function in 
$ C^{\infty}(\overline{\R_+}),$ $\hb(t)\ge0,$ $\supp 
\hb\subset[\frac12,1]$, normalized by $\int_0^1 \hb(t)dt=1$. The scaled 
function 
$h_r(t)=r^{-1}\hb(t/r)$ is normalized in such a way that 
$\int_0^\infty h_r(t)dt=1.$ 
Further on, for $x\in \Ct\Kb,\,t< r=\dn(x) $, we denote by 
$S_t(x)$ the sphere 
$\{y:|y-x|=t\}$, 
  and by $B_t(x)$ the corresponding open ball;
   they do not touch $\Kb$, moreover, they are on a controlled 
   distance from $\Kb$.

Next, for our elliptic operator
 \begin{equation}\label{operator}
 \Lc(x,\partial_x)=-\sum_{\jt,\jt'}\partial_\jt 
 a_{\jt\jt'}(x)\partial_{\jt'}\equiv -\nabla\cdot(\ab(x)\nabla)
 \end{equation}  
 in the unit ball $\Om\equiv B_1(\mathbb{O})$ containing $\Kb$, we  construct 
  the averaging kernel $K(x,y)$. It will 
    act as a replacement, for the operator $\Lc$, of the mean 
    value 
kernel, usual for the Laplacian: 
  namely, for any point $x\in \Ct\Kb$ and a function $\f(y)$, 
  $\Lc$-harmonic
  in the ball $B_r(x),$ $\f\in C(\overline{B_r(x)}) $, 
  $r=r(x)\le\dn(x)$,
   the following  representation  holds:
\begin{equation*}%\label{Int.repres}
  \f(x)=\int_{B_r(x)}K(x,y) \f(y) dy.
\end{equation*}

 The construction of $K(x,y)$ is performed in the following way.
 For a fixed point $z\in\Ct\Kb$, we 
denote by $G_{z,t}(x,y)$ the Green function for the Dirichlet 
problem for 
$\Lc$ in the ball $B_t(z),$ so, in this notation, variables 
$(z,t)$ in the subscript indicate the 
domain, the ball, 
where the Green function is considered, while, further, $(x,y)$ are the 
variables on 
which the value of the Green function depends. We are interested, especially, in $G_{x,t}(x,y)$, this means, the value of the Green function at the centerpoint of the sphere. In 
particular, we have
$G_{x,t}(x,y)=0$ on the sphere  $|y-x|=t$.  Since the operator 
$\Lc$ is Hermitian,
 we have $G_{x,t}(x,y)=G_{x,t}(y,x)$. 
 
 For $t\in[r/2,r]$, $y\in S_t(x)$, we denote 
by $\pmb{\pmb{\n}}_y(t)$ the external $\Lc$-conormal to the sphere 
$S_t(x)$ at the point $y\in 
S_t(x);$ namely, for an operator $\Lc$ in \eqref{operator},
 the components of the conormal 
 are 
$(\pmb{\n}_y(t))_{\jt}=\sum_{\jt'}a_{\jt,\jt'}(y)\nb_{\jt'}(y)$, where 
$\nb_{\jt'}(y)$ are the 
components of the Euclidean normal to $S_t(x)$  at the point $y$ 
(note that the 
conormal vector is not normalized).

By the classical representation formula for   solutions of 
elliptic equations, the 
following equation is valid
for a function $\f$, as long as $\f\in C(\overline{B_t(x)})$, $\Lc 
\f=0$
in ${B_t(x)}$,

\begin{equation}\label{Repr.sphere}
 \f(x)= 
 \int_{S_t(x)}\partial_{\pmb{\pmb{\n}}_y}G_{x,t}(x,y)\f(y)d 
 \s_t(y),
\end{equation}
where $d \s_t$ is the surface element on the sphere $S_t(x)$. 
We multiply \eqref{Repr.sphere} by $h_r(t),$ $r=\dn(x)$ and 
integrate in $t$.
By our normalization, the left-hand side in \eqref{Repr.sphere} 
remains  equal to
$\f(x)$ after integration,  while the right-hand side takes the 
form of the integral over the 
spherical annulus
$\Rc_r(x)=B_{r}(x)\setminus B_{\frac{r}2}(x)$:
\begin{equation}\label{Repr.ball}
  \f(x)=\int_{\Rc_r(x)} 
  h_r(|x-y|)\partial_{\pmb{\pmb{\n}}_y}G_{x,|x-y|}(x,y)\f(y)d y, \, r=d_0(x).
\end{equation}
It is the kernel in \eqref{Repr.ball} that  will be accepted as 
$K(x,y)$: we set 
 \begin{equation*}% \label{kernel}
K(x,y)=h_{\dn(x)}(|x-y|)\partial_{\pmb{\pmb{\n}}_y}G_{x,|x-y|}(x,y), 
\end{equation*}
 for $y\in \Rc_r(x) $, and  $ K(x,y)=0$ 
otherwise.

 As follows from   \cite{Widman}, the function $K(x,y)$ satisfies
 \begin{equation*} |K(x,y)|\le c 
\dn(x)^{-\Nb}, \, y\in\Rc_r(x). \end{equation*} 
 In the following, we will also need estimates for the derivatives 
of the kernel $K(x,y)$ in the same domain $y\in\Rc_r(x)$, namely,
 \begin{equation}\label{Deriv.K.1}
  |\nabla_x K(x,y)|\le c\dn(x)^{-\Nb-1},
\end{equation}
 and
  \begin{equation}\label{Deriv.K.2}
  |\nabla^2_x K(x,y)|\le c\dn(x)^{-\Nb-2}.
\end{equation}
These estimates are established further on.
\subsection{Estimates of derivatives of $K(x,y)$. First order 
derivatives}\label{subsect.Kx}
 Since the  dependence 
of $K$ on $x$ is 
rather 
implicit, the calculations of its derivatives are rather cumbersome and 
involve multiple applications of the chain rule. We would like to explain at this point, somewhat informally, how this calculation goes through and which 
derivatives of the Green function may appear;  this facilitates better understanding of rigorous 
reasoning afterwards.

The easiest term to handle is the factor $h_{\dn(x)}(|x-y|)$; its derivatives are calculated directly. More trouble is 
caused by the factor $\partial_{\pmb{\pmb{\n}}_y}G_{x,|x-y|}(x,y)$. First, note that when the position of the point $x$ 
changes infinitesimally, the sphere $\partial B_{|x-y|}(x)$ has its center moving, the sphere still passing through $y$,  
therefore the  conormal vector ${\pmb{\pmb{\n}}_y}$ changes its length and direction, and in evaluating this change, the 
derivatives of $G$ in $x$ and $y$ appear. Next, the size of the ball $B_{|x-y|}(x)$ changes, and this change   should also be
taken into account when differentiating in $x$. Finally, we must notice that when the point $x$ moves, the operator 
$\Lc$ changes. To see this, we take an infinitesimally moved point $\x=\x(\vs)$, the  new position of $x$,  $\x(0)=x$,  and consider a 
linear conformal transformation $\Tt(\vs)=\Eb+O(\vs)$ in $\R^{\Nb}$ mapping the 'old' ball $B^{(0)}= B_{|x-y|}(x)$ onto the 'new' 
ball $B^{(\vs)}= B_{|\x-y|}(\x)$. Under this transformation, the operator $\Lc$ in the new ball $B^{(\vs)}$ transforms 
into a new operator   $\Lc(\vs)$  in the initial ball $B^{(0)}$ with coefficients $\ab(x,\vs)=\ab(\Tt(\vs)x)$, depending 
on the parameter $\vs$; in other words, instead of considering the 'old operator $\Lc$' in a 'new ball', we consider here the new operator $\Lc(\vs)$ in the old ball. Therefore, when we differentiate the Green function in $x$, we need to keep in mind this change of the 
operator and thus, for a fixed ball, the derivatives $\partial_{\vs} G(x,y;\vs)$,  $\partial_{\vs}\partial_y 
G(x,y;\vs)$ and  $\partial_{\vs}\partial_x 
G(x,y;\vs)$ arise, where $G(x,y;\vs)$ is the Green function of the operator $\Lc(\vs)$ in $B^{(0)}$.

 Now we present rigorous calculations rendering concrete these hints.

   We start with evaluating derivatives of an 
important 
auxiliary vector-function. 
$\mb(x,z)=(x-z)\dn(x)^{-1}, $ for $x,z\in\Ct\Kb,$ under the 
condition 
$|x-z|\le\dn(x).$   

First, of course,  $\nabla_z \mb(x,z)=\pmb{\Et}\dn(x)^{-1},$ 
therefore, $|\nabla_z 
\mb(x,z)|=\Nb^{\frac12}\dn(x)^{-1}$.  The calculation of 
the 
$x$-derivative is a little bit 
more laborious: 
 \begin{gather*}
  \nabla_{x}\mb(x,z)= 
  (\dn(x)^{-1})_x'\otimes (x-z)+\dn(x)^{-1}\pmb{\Et}=\\\nonumber
  \dn(x)^{-2}\nabla_{x}\dn(x)\otimes(z-x)+\dn(x)^{-1}\pmb{\Et},
\end{gather*} therefore, \begin{equation}\label{Imp.Estim} 
  |\nabla_x \mb(x,z)|\le c \dn(x)^{-1}\le c'd(x)^{-1},
\end{equation} 
and, similarly,

\begin{equation}\label{m_zz}
  |\nabla^2_{xx}\mb(x,z)|\le c d(x)^{-2}.
\end{equation}

Now we  estimate
 first order derivatives of the function $K(x,y) $ with 
 respect to the $x$-variable.
 We fix a point  $x\in \Ct\Kb$ 
and consider $\x$ in the small ball $ B_{\dn(x)/4}(x),$  $y\in 
S_t(x)$ for 
$t=|x-y|$. The function $K(\x,y)$ depends on $\x$ in a complicated way; this dependence involves 
the change of the regularized  distance $d_0(\x)$, which leads to the change of the radius of 
the ball,  and also change of  the direction of the normal and conormal vectors at $y$ as $\x$ 
moves. Thus we encounter many terms requiring detailed analysis. 

 By the definition of $K(\x,y),$ we have 
\begin{gather}\label{Deriv.K.1.1} 
  K'_\x(\x,y)=(h_{\dn(\x)}(|\x-y|))'_{\x} 
  \partial_{\pmb{\pmb{\n}}_y(t)}G_{\x,|\x-y|}(\x,y)+\\\nonumber
  h_{\dn(\x)}(|\x-y|)(\partial_{\pmb{\pmb{\n}}_y(t)}G_{\x,|\x-y|}(\x,y))'_{\x}.
\end{gather} As follows from  \eqref{Imp.Estim}  and the obvious 
estimate 
\begin{equation}\label{Deriv.K.2'}
\left|(h_{\dn(\x)}(|\x-y|))'_{\x}\right|\le c \dn(\x)^{-2},
\end{equation}
 for the first term in 
\eqref{Deriv.K.1.1}, we have \begin{equation}\label{Deriv.K.2.1} 
  \left|(h_{\dn(\x)}(|\x-y|))'_\x 
  \partial_{\pmb{\pmb{\n}}_y(t)}G_{\x,|\x-y|}(\x,y)\right|\le c 
  \dn(\x)^{-\Nb-1}.
\end{equation}Next we consider the second term in 
\eqref{Deriv.K.1.1}. As explained earlier,  while 
the point $x$, the center of the sphere, moves to the position $\x$, 
under such 
movement, the point $y$ remains at its place, but the sphere 
$S_{|x-y|}(x)$ moves to the new position $S_{|\x-y|}(\x)$, and this causes the change of the direction 
of the tangent plane at $y$ and thus of the conormal vector
$\pmb{\pmb{\n}}_y$. 

To evaluate this change of direction,  when calculating  derivatives 
with 
respect to $\xi$ variable,  it is convenient to parametrize  the point $\x$ in a special way. We 
need some elementary geometry here.  Denote by $\pmb{\Sc}$ the unit sphere 
with 
center at the origin $\mathbb{0}$, consider the point $\la_0=\frac{x-y}{|x-y|}\in \pmb{\Sc}$, and let 
 $\pmb{\om}$ be a neighborhood of $\la_0$ in $\pmb{\Sc}:$ 
$\pmb{\om}=\{\la=\frac{\xi-y}{|\xi-y|}: 
|\xi-x|\le \frac14\dn(x)\}.$ 

We set ${\ta}=|\xi-y|$  and consider the change of the operators 
of the conormal 
derivative, $\tilde{\partial}=\partial_{\pmb{\pmb{\n}}_y({\ta})}$, 
as $\x$ moves. The conormal 
derivatives 
can be understood as directional derivatives along the vectors 
$\pmb{\pmb{\n}}_y(t),$ not 
necessarily the unit ones. Since the position of the sphere 
$S_t(\xi)$ is 
determined by the parameter 
${{\varsigma}}=\frac{\x-y}{d_0(\x)}\in\pmb{\om}$, we 
include this parameter into 
the notation, suppressing temporarily $t$: 
$\pmb{\pmb{\n}}(y,{\varsigma}):=\pmb{\pmb{\n}}_y(t)$. In the lucky case, when 
$\pmb{\pmb{\n}}(y,{\varsigma})\|\pmb{\pmb{\n}}(y,\la)$, the 
reasoning is quite simple since there is no change of the 
direction of the conormal vector. We 
consider the 
generic case  
$\pmb{\pmb{\n}}(y,{\varsigma})\nparallel\pmb{\pmb{\n}}(y,\la).$ 
Let us draw the 
two-dimensional plane $\prod$ through the point $y$ and  the 
vectors  
$\pmb{\pmb{\n}}(y,{\varsigma}),\pmb{\pmb{\n}}(y,\la)$. This plane 
cuts the sphere $S_{{\ta}}(x)$ along 
the 
circle that we denote by $\vp$. We  consider the intersection 
$\prod\cap 
B_{{t}}(\xi)$.
 It is a two-dimensional disk.

Denote by $\n_{\x,y}$ the unit interior normal vector (the 
Euclidean one, note the font difference!) to $\vp$.  
By definition, the Green function $G_{\xi,|\xi-y|}(\xi,y)$ 
vanishes for $y\in 
S_{{\ta}}(\xi)$, in particular, on $\vp$. 
Therefore, 
\begin{gather}\label{der1}
  \partial_{\pmb{\pmb{\n}}(y,\la)}G_{\xi,|\xi-y|}(\xi,y)=\\\nonumber
  |\pmb{\pmb{\n}}(y,{\la})|^{-1}\partial_{\n_{\xi,y}}G_{\xi,|\xi-y|}(\xi,y)\cos(\widehat{{\n}_{\xi,y},\pmb{\pmb{\n}}(y,\la)}),
\end{gather}
and 
\begin{gather}\label{der2}
  \partial_{\pmb{\pmb{\n}}(y,{\varsigma})}G_{\xi,|\xi-y|}(\xi,y)=\\\nonumber
  |\pmb{\pmb{\n}}(y,{\varsigma})|^{-1}G_{\xi,|\xi-y|}(\xi,y)\cos(\widehat{\n_{\xi,y},\pmb{\pmb{\n}}(y,\la)}),
\end{gather}
where $\widehat{\n,\n'}$ denotes the angle between the vectors 
$\n,\n'$.
Therefore, the ratio of the derivatives in \eqref{der1} and 
\eqref{der2}, which we 
denote by 
$\m(y,\la,{\varsigma})$, satisfies
\begin{equation}\label{6}
 \m(y,\la,\la)=1,\, |\nabla_{{\varsigma}}\m(y,\la,{\varsigma})|\le 
 C. 
\end{equation}
Now we are able to conclude the calculation of the change of the 
kernel $K$ under
 the change of the point $x$, the center of the ball.
Having $\xi$ in a neighbourhood of $x$, due to the definition of 
the kernel $K$, we 
have
\begin{gather}\label{7}
  \partial_{\xi}K(\xi,y) =\partial_{\xi}[h_{\dn(\xi)}
  (|\xi-y|)\partial_{\pmb{\pmb{\n}}(y,{\varsigma})}G_{\xi,|\xi-y|}(\xi,y)]=
  \\\nonumber
  \partial_{\xi}(h_{\dn(\xi)})(|\xi-y|))\partial_{\pmb{\pmb{\n}}(y,{\varsigma})}G_{\xi,|\xi-y|}(\xi,y)+\\\nonumber
   h_{\dn(\xi)}\partial_{\xi}\left(\partial_{\pmb{\pmb{\n}}(y,{\varsigma})}G_{\xi,|\xi-y|}(\xi,y)\right)\equiv I_1+I_2.
\end{gather}
For the first term in \eqref{7}, we use our estimate 
\eqref{Deriv.K.2'} for the 
derivative of $h$,
\begin{equation*}
  |\partial_{\xi}(h_{\dn(\xi)}(|\xi-y|))|\le c \dn(x)^{-2}.
\end{equation*}
For the derivatives of the Green function, in the first term on the 
right in 
\eqref{7}, we use
   estimate \eqref{A.R ball}, which gives 
  $|\partial_{\pmb{\pmb{\n}}(y,{\varsigma})}G(\xi, y)|\le c 
  \dn(x)^{1-\Nb}$,
   Therefore, for the first term $I_1$ in \eqref{7}, we have
  \begin{equation}\label{8}
    |I_1|\le C \dn(x)^{-1-\Nb}.
  \end{equation}
 Next, we estimate $I_2$.  For the last factor, i.e., for the 
  second derivative of 
  the Green function,
  we have
  \begin{gather}\label{9}
    \partial_{\xi}(\partial_{\pmb{\pmb{\n}}(y,{\varsigma})} 
    G_{\xi,|\xi-y|}(\xi,y))=\\\nonumber
    \partial_{\xi}\left[\m(y,\la,{\varsigma})|_{{}_{{\varsigma}=
    \frac{\xi-y}{\dn(x)}}}\partial_{\pmb{\pmb{\n}}(y,\la)}G_{\xi,|\xi-y|}(\xi,y)\right]=\\\nonumber
    =\partial_{\xi}\m(y,\la,{\varsigma})|_{{}_{{\varsigma}
    =\frac{\xi-y}{\dn(x)}}}\partial_{\pmb{\pmb{\n}}(y,\la)}G_{\xi,|\xi-y|}(\xi,y)+\m(y,\la,{\varsigma})\partial_{\xi}\partial_{\pmb{\pmb{\n}}(y,\la)}G_{\xi,|\xi-y|}(\xi,y)=\\\nonumber
    I_3+I_4.
  \end{gather}
  The term $I_3$ is again estimated using \eqref{Deriv.K.1.1} and 
  the  bounds for 
  the first
   and second order derivatives of the Green function in \eqref{WidmanEst}.
  This gives   
  \begin{equation*}
    \left|\partial_{{\varsigma}}\m(y,\la,{\varsigma})\partial_{\x}\left[\frac{\xi-y}{\dn(\x)}\right]\right|\le 
    c\dn(x)^{-1},
  \end{equation*}
  and, therefore, 
  \begin{gather}\label{9'}
    \left|h_{\dn(\x)}(|\xi-y|)\partial_{{\varsigma}}\m(y,\la,{\varsigma})
    \partial_{\x}\left[\frac{\xi-y}{\dn(\x)}\right]\partial_{\pmb{\pmb{\n}}(y,\la)}G_{\xi,|\xi-y|}(\xi,y)\right|\le\\\nonumber 
    c \dn(x)^{-1}\cdot\dn(x)^{-1}\cdot     
    \dn(x)^{1-\Nb}=c\dn(x)^{-1-\Nb}.
  \end{gather}
  The most troublesome is the evaluation of the term $I_4$, since 
  it requires 
  tracing  the behavior of derivatives
   of the Green function under a rotation of the co-ordinates 
   system.
   
   Consider the family of linear conformal mappings $\y\mapsto \Tt({\varsigma})\y$ 
   in 
   $\R^{\Nb}$, 
   transforming $x-y$ to $\xi-y$, depending smoothly on 
   $\varsigma$
    (the parameter ${\varsigma}$ is concealed in $\xi$),
   \begin{equation*}
     \Tt({\varsigma})=\frac{|{\varsigma}|}{|\la|}\Bt({\varsigma}),
   \end{equation*}
   where $\Bt({\varsigma})$ is an orthogonal transformation 
   depending smoothly on 
   ${\varsigma}$, with $\mathrm{T}(\la)=\Et$.
   Next, we denote by $g(u,v,{\varsigma})$ the rotated Green 
   function, namely, its 
   image under the transformation $\Tt({\varsigma})$
   of the ball $B(x,t)$ to the ball $B(\xi,{\ta})$:
   \begin{gather}\label{10}
     g(u,v,{\varsigma}):=\\\nonumber 
     G_{\Tt\left({\varsigma}\right)(x-y)+y,|{\varsigma}|\dn\left(\Tt({\varsigma})(x-y)+y\right)}(\Tt({\varsigma})(v-y)+y, 
     \Tt({\varsigma})(u-y)+y).
   \end{gather}
   It follows from \eqref{10} that the  coefficients 
   $a_{\jt,\jt'}(.,{\varsigma})$ of
    the operator $\Lc({\varsigma})$ obtained from $\Lc$ by the 
    transformation 
    $\Tt({\varsigma})$ depend $C^{m}$-smoothly on the parameter 
    ${\varsigma}$. 
    
    By the multiple usage of the chain rule, we express the 
    derivatives of the function 
    $g(y+T({\varsigma})(x-y),y,{\varsigma})$
     via the derivatives of $g(\xi,y,{\varsigma})$; the latter is 
     the Green 
     function in the ball $B_{{\ta}}(\xi)$
      for the transformed operator $\Lc({\varsigma})$. We obtain
     \begin{gather}\label{11}
     \partial_{\xi}\partial_{\pmb{\pmb{\n}}(y,\la)}g(y+\Tt({\varsigma})(x-y),y,{\varsigma}))
     =\\\nonumber
       \partial_{\pmb{\pmb{\n}}(y,\la)}\partial_{\xi} 
       g(y+\Tt({\varsigma})(x-y),y,{\varsigma})
       =\partial_{\pmb{\pmb{\n}}(y,\la)}\partial_{{\varsigma}}g(y+\Tt({\varsigma})(x-y))\partial_{\xi}{\varsigma}|_{{}_{\varsigma=\frac{\x-y}{\dn(x)}}}=     
      \\\nonumber
       \partial_{\pmb{\pmb{\n}}(y,\la)}\left( \partial_\x 
       g(y+\Tt(\varsigma)(x-y),y,\varsigma\right)\partial_{\varsigma}\Tt(\varsigma)(x-y)
       +\\\nonumber\partial_{\varsigma}g(y+\Tt(\varsigma)(x-y),y,\varsigma)\partial_\x{\varsigma}=\\\nonumber
      \partial_{\pmb{\pmb{\n}}(y,\la)}\partial_{\x}g(y+\Tt(\varsigma)(x-y),y,\varsigma)\partial_{\varsigma}\Tt(\varsigma)(x-y)+\\\nonumber
      \partial_{\x}g(y+\Tt(\varsigma)(x-y),y,\varsigma)\partial_{\pmb{\pmb{\n}}(y,\la)}\partial_{\varsigma}\Tt(\varsigma)(x-y)\partial_{\xi}{\varsigma}+
       \\\nonumber 
      \partial_{\xi}g(y+\Tt({\varsigma})(x-y),y,{\varsigma})\partial_{\varsigma}\Tt(\varsigma)(x-y)\partial_\n\partial_{\xi}\varsigma+\\\nonumber
\partial_{\pmb{\pmb{\n}}(y,\la)}\partial_{\varsigma}g(\x,y,\varsigma)\partial_{\x}\varsigma+\partial_{\varsigma}g(\x,u,\varsigma)\partial_{\pmb{\pmb{\n}}(y,\la)}\varsigma|_{{}_{\varsigma=\frac{\x-y}{\dn(x)}}}\\\nonumber
   =\partial_{\pmb{\pmb{\n}}(y,\la)}\partial_{\x}g(y+\Tt(\varsigma)(x-y),y,\varsigma)\partial_{\varsigma}\Tt(\varsigma)(x-y)\partial_{\x}\varsigma\\\nonumber
       +\partial_{\xi}g(y+\Tt({\varsigma})(x-y),y,{\varsigma})\partial_{\pmb{\pmb{\n}}(y,\la)}\partial_{{\varsigma}}\Tt({\varsigma})(x-y)\partial_{\xi}{\varsigma}+
       \\\nonumber
      \partial_{\pmb{\pmb{\n}}(y,\la)}\partial_{{\varsigma}}g(\xi,y,{\varsigma})|_{{{\varsigma}=\frac{\xi-y}{\dn(x)}}}.  
      \end{gather}
     When obtaining \eqref{11}, we applied the fact that 
     $\partial_{\pmb{\pmb{\n}}(y,\la)}\partial_{\xi}\frac{\xi-y}{\dn(x)}=0.$
     
     We collect now the estimates which we use to treat 
     \eqref{11}:
     \begin{gather}\nonumber
       |\partial_{\pmb{\pmb{\n}}(y,\la)}\partial_{\xi}g(y+\Tt(\la)(x-y), 
       y,{\varsigma})|\le C 
       \dn(x)^{-\Nb}; \\\label{12}
       |\partial_{{\varsigma}}\Tt({\varsigma})(x-y)|\le c|x-y|\le 
       c\dn(x); \, 
       \left|\partial_{\xi}\frac{\xi}{\dn(x)}\right|\le c 
       \dn(x)^{-1};\\\nonumber
       |\partial_{\pmb{\pmb{\n}}(y,\la)}\partial_{{\varsigma}}\Tt({\varsigma})(x-y)|\le 
       C; \, 
       |\partial_{\pmb{\pmb{\n}}(y,\la)}\partial_{{\varsigma}}g(\xi,y,{\varsigma})|\le 
       C 
       \dn(x)^{1-\Nb}.
      \end{gather}
     After their substitution to \eqref{11},  we obtain 
     \begin{equation}\label{13a}
       \left| 
       \partial_{\xi}\partial_{\pmb{\pmb{\n}}(y,\la)}g(y+\Tt({\varsigma})(x-y),y,{\varsigma})\right|\le 
       c \dn(x)^{-\Nb}.
     \end{equation}
     Taken together with \eqref{7},\eqref{8},\eqref{9},\eqref{9'}, 
     this gives us 
     the required  estimate 
     for the derivative of the kernel  $K(\xi,y)$ for $\x=x$:
     \begin{equation}\label{derivative K}
       \left| \partial_{x}K(x,y)\right|\le c \dn(x)^{-1-\Nb}.
     \end{equation}
      Before we proceed, we review which estimates of derivatives of the Green function 
      $G(x,y;\varsigma)$ we need. In estimating derivatives of the kernel $K(x,y),$ we applied the chain 
      rule several times. In this calculations, some factors appeared related to  
      variables changes and they  did not involve the Green function, but only dealt with geometrical 
      quantities.  The only factors containing  derivatives of the Green function were 
      derivatives  in $x$,  accounting for the movement of the point $x$, see \eqref{9'}, and 
      second derivatives in $x,y$ and $y,\varsigma$ involved in  the evaluation of the contribution of the  
      change of direction of the conormal derivative, see \eqref{13a}, as well as the change of operator $\Lc$. Combined with other 
      terms, they give estimate \eqref{derivative K}.  We stress here that we use here 
      derivatives of the Green function for a ball with radius $d_0(x)$ only at the  points 
      $x,y,$ such that $x$ lies at the center of the ball and $y$ lies at its boundary, so, 
      the distance between $x,y$ is controlled from below by the radius of the ball.

     \subsection{Estimating the second-order derivatives, 
     $K^{''}_{xx}(x,y)$}
     We will also need estimates for the second-order derivatives of the 
     kernel $K(x,y)$. 
     The calculations go
     essentially in the same way as for the first derivative, 
     however they are 
     considerably more cumbersome, so we explain only their 
     structure. Again, multiple application of the chain rule leads to a number of factors of 
     geometrical nature, not depending on the Green function, as well as derivatives of the 
     Green function. Of these derivatives, $\partial^2_{xx}G(x,y;\varsigma)$ reflects the contribution of 
     the movement of the point $x$, while $\partial^3_{xx y}G(x,y;\varsigma)$, reflects the change of the 
     direction of the conormal derivative. Additionally, as explained in the previous subsection, there are terms 
     reflecting the change of the operator $\Lc$ when passing from the ball $B_{|x-y|}(x)$ to the ball $B_{|\x-y|}(\x)$;
     here, additionally the derivatives $\partial^2_{\varsigma\varsigma}G(x,y;\varsigma)$,  
     $\partial^3_{\varsigma\varsigma x}G(x,y;\varsigma)$,  $\partial^3_{\varsigma\varsigma y}G(x,y;\varsigma)$,  
     $\partial^3_{\varsigma xy}G(x,y;\varsigma)$ appear.
     
     In more detail, after differentiating \eqref{11}, we obtain the sum of ten 
     terms; we present 
     them, recalling the notation \eqref{10} and setting symbolically for shorthand
     $\Xi:=(y+\Tt({\varsigma})(\x-y),y,{\varsigma})$ and 
     $\Psi=(\x,y,{\varsigma})$:
     \begin{gather}\label{13}K^{''}_{\xi,\xi}(\xi,y)=\partial_{\pmb{\pmb{\n}}(y,\la)}g(\Xi)A_1(\Psi)+\partial_{\pmb{\pmb{\n}}(y,\la)}\partial_{\xi}g(\Xi)A_2(\Psi)+
     \\\nonumber 
     \partial_{\xi}g(\Xi)A_3(\Psi)+\partial_{\xi}\partial_{\xi}g(\Xi)A_4(\Psi)+\partial_{\pmb{\pmb{\n}}(y,\la)}\partial_{{\varsigma}}g(\Xi)A_5(\Psi)
     +\\\nonumber 
     \partial_{\xi}\partial_{{\varsigma}}g(\Xi)A_6(\Psi)+\partial_{\pmb{\pmb{\n}}(y,\la)}\partial^2_{{\varsigma}} 
     g(\Xi)A_7(\Psi)+ 
     \partial^2_{\xi}\partial_{\pmb{\pmb{\n}}(y,\la)}g(\Xi)A_8(\Psi)
     +\\\nonumber 
     \partial_{\pmb{\pmb{\n}}(y,\la)}\partial_{\xi}\partial_{{\varsigma}} 
     g(\Xi)A_9(\Psi)+\partial_{\pmb{\pmb{\n}}(y,\la)}\partial^2_{{\varsigma}}A_{10}(\Psi).
      \end{gather}
      Here, in  \eqref{13}, as before, 
      $\xi=y+\Tt({\varsigma})(x-y)$, 
      ${\varsigma}=\frac{\xi-y}{\dn(\x)}$, and the expressions
      $A_1$ to $A_{10}$ are combinations of the functions 
      $h_{\dn(x)}(|x-y|)$, 
      $\m({\varsigma})$, $\Tt({\varsigma})(x-y)$ and their 
      derivatives, 
      structurally similar 
      to the expressions presented in \eqref{11}, where the first 
      order derivatives 
      of $K(\xi,y)$ were being treated, as well as derivatives of $G$. These expressions can 
      be 
      estimated
      using relations \eqref{12}. We keep in mind that, due to our 
      choice of 
      $x,\xi,y$, we have
      \begin{equation}\label{same order}
       |\xi-y|\ge \frac14 |x-y|=\frac14 \dn(x)
      \end{equation}
      
      To make the explanation of these  calculations more 
      transparent, we introduce 
      the  following generic notations.
       By $\pmb{\de} g$ we denote any of first order derivatives of a 
       function $g(\Xi)$ in the variables $\xi,\n$; 
       by $\pmb{\de} ^2 g$ we denote any of second order derivatives
       in these variables, and, similarly, $\pmb{\de}^3g$ stands for 
       third order 
       derivatives.
       Relations \eqref{same order} imply 
       \begin{equation*}%\label{14}
         |\pmb{\de}  g(\Xi)|, |\partial_{{\varsigma}}\pmb{\de}  g(\Xi)|, 
         |\partial^2_{{\varsigma}}\pmb{\de}  g(\Xi)|\le C \dn(x)^{1-\Nb};
       \end{equation*}
       \begin{equation*}%\label{15}
       |\pmb{\de} ^2 g(\Xi)|, |\partial_{{\varsigma}} g(\Xi)|\le 
       c\dn(x)^{-\Nb};
       \end{equation*}
       \begin{equation*}%\label{16}
       |\pmb{\de} ^3 g(\Xi)|\le c \dn(x)^{-1-\Nb}.
       \end{equation*}
       
       On the other hand, we estimate the quantities 
       $A_1$--$A_{10}$ using 
       \eqref{12}, \eqref{6}, \eqref{Deriv.K.1.1}, 
       \eqref{Deriv.K.2.1}, 
       \eqref{Deriv.K.2'};
       their  substitution into \eqref{13} gives
       \begin{equation*}%\label{13?}
         |K_{\xi,\xi}(\xi,y)|\le c d_0(x)^{-2-\Nb},
       \end{equation*}
       in particular, for $\xi=x$, we have
       \begin{equation}\label{17}
         |K_{xx}(x,y)|\le c d_0(x)^{-2-\Nb}.
       \end{equation}

\section{The approximating function; proof of Theorem 1.2}\label{sect. appr.}
 In this section, we construct the  extension
function $f_0(x)$ and establish its 
properties, with further construction of the approximating 
function $v_\de$. 
\subsection{An estimate for the local 
approximation}\label{lem.local.appr}
In order to  study  properties of the extension function $f_0$, see \eqref{function f}, we 
need an estimate  
for derivatives of the  local approximations $\F_{x,\de}$.
\begin{lem}\label{Lem.7.1}Let $f(x),$ $x\in \Kb$ satisfy the 
conditions of Theorem 
\ref{MainTheorem}, with  $\rb\ge1.$ Let $x\in\Kb$
 and $\de>0$. Then for $y\in B_{\frac{\de}{2}},$
\begin{equation}\label{Est.F}
  |\grad_y \F_{x,\de}(y)| \le C,
\end{equation}
where the constant $C$ may depend  on $\Kb$, $\Lc$ and  $f$, but for fixed $\Kb$, $\Lc$,  
$f$,
does not depend on 
$\de$.
\end{lem} 
\begin{proof}Choose the integer $N$ so that 
$2^{N-1}\de<\diam(\Kb)\le 2^{N}\de$. 
Then
\begin{equation*}
  \F_{x,\de}(y)=\F_{x,2^N\de}(y)-\sum_{k=1}^N\left( 
  \F_{x,2^k\de}(y)-\F_{x,2^{k-1}\de}(y)\right).
\end{equation*}
We apply the classical property for the gradient of solutions of 
elliptic equations 
(see, e.g., \cite{Widman}): if a function $\Psi(y)$
 is a solution of the elliptic equation $\Lc \Psi(y)=0$ in the 
 ball $B_R(x),$ then 
 in the smaller concentric
   ball $y\in B_{\frac{R}{2}}(x),$
 the estimate holds

 \begin{equation*}
   |\grad \Psi(y)|\le C R^{-1}\|\Psi\|_{C(B_R(x))},
 \end{equation*}
 with constant $C$ not depending on $R$.\\
 Due to our choice of $N$, for $y\in B_{\frac{\de}{2}}(x)$, 
 \begin{equation*}
  |\grad_y \F_{x,2^N\de}(y)|\le C,  
 \end{equation*}
 since the function $\F_{x,2^N\de}$, which serves as a 'local 
 approximation' with error $2^{N}\de$ on the 
 whole of $\Kb$, is bounded in $B_{2^N\de}(x)$ uniformly in 
 $x\in\Kb$.
  
  We may assume  $\de<1$. Then Condition \eqref{F2}, for $k\ge1$ 
  implies
  \begin{equation*}
    \|\F_{x,2^{k-1}\de}-\F_{x,2^{k}\de}\|_{C(B_{2^{k-1}\de}(x))}\le
         C \left(2^{k\rb}\de\right)\de^\rb\om(2^k\de).
     \end{equation*}
     Therefore,
     
     \begin{gather*}
       \|\grad\left(\F_{x,2^{k-1}\de}-\F_{x,2^{k}\de}\right)\|_{C\left(\overline{B_{\frac{\de}{2}}(x)}\right)}\le\\\nonumber
        C (2^{k}\de)^{-1}\times 2^{k\rb}\de^{\rb}\om(2^k)=
        c 2^{k(\rb-1)}\de^{\rb-1}\om(2^k\de),\, \rb\ge 1.
     \end{gather*}
     It follows now that
  \begin{gather*}
 \|\nabla \F_{x,\de}
  \|_{C(\overline{B_{\frac{\de}{2}}(x)})}\le 
  \\\nonumber
  \sum_{k=1}^{N}\|\grad(\F_{x,2^{k-1}\de}-\F_{x,2^{k}\de})\|_{C(\overline{B_{\frac{\de}{2}}(x)})}+
 \|\grad\F_{x,2^N\de}\|_{C(\overline{B_{\frac{\de}{2}}(x)})}\le 
 \\\nonumber 
c\sum_{k=1}^{N}2^{k(\rb-1)}\de^{\rb-1} \om(2^k\de)+c\le 
\de^{\rb-1}\int_0^N 
2^{s(\rb-1)}\om(2^s\de)ds+c=
\\\nonumber c\de^{\rb-1}\int_1^{2^{N}}t^{\rb-1}\om(\de 
t)\frac{dt}{t}+c=
c\de^{\rb-1}\left(\de^{-1}\right)^{\rb-1}\int_0^{2^N\de}\tau^{\rb-2}\om(\z)d\tau+c\\\nonumber
\le C(\diam \Kb)^{\rb-1}\int_0^{2\diam\Kb}\om(\tau)\tau^{-1}d\z+c\le 
C.
  \end{gather*}  
\end{proof}
 \subsection{Estimating $\Lc f_0(x)$}\label{Sect.estimat.f0(x)}
       Now, using our estimates  for derivatives of the kernel 
       $K(x,y)$, 
   obtained in the previous section, we establish estimates for 
   the function $f_0(x)$ constructed 
       in Section 3.1, see \eqref{function f},
        and for the result of the action of the operator $\Lc$ on 
        this function.
  For a fixed point $x_0\in\Ct \Kb$, we consider
 the open cube $Q$  in the Whitney cover $\Qc$,  whose closure    contains  $x_0$. Let $x_{Q}$ be 
 the point in $\Kb$, closest to the
 center of this cube  (or one of such points).
   Recall that the function $\fT_0(x)$ (defined in \eqref{f-tilde}) 
   equals $0$ on the boundary of $Q$.
 By construction, $\Lc_x\F_{x_Q,2\de(Q)}(x)=0$, for $x$ 
   in the ball $B_{2\de(Q)}(x_Q)$, in particular,
   this holds in a small  neighborhood of the ball 
   $\overline{{B_{\dn(x_0)}(x_0)}}$. Therefore, for $x$ in a small 
   neighborhood of 
   the point $x_0$, 
   we obtain, recalling the definition of the kernel $K(x,y)$ 
   (which acts as an 
   $\Lc$-replacement  for the mean value kernel):
\begin{gather}\label{18.p11}
  |\Lc_x f_0(x)|= |\Lc_x(f_0(x)-\F_{x_Q,2\de(Q)}(x))|=\\\nonumber
  \left|\Lc_x\left(\int_{\R^\Nb}  
  \fT_0(y)K(x,y)dy-\int_{\R^{\Nb}}\F_{x_Q,2\de(Q)}(y)K(x,y)dy\right)\right|=\\\nonumber
\left| \Lc_x\left( 
\int_{\R^\Nb}(\fT_0(y)-\F_{x_Q,2\de(Q)}(y))K(x,y) dy 
\right)\right|\le \\\nonumber
C\int\limits_{{B_{\dn(x_0)+\ve}(x_0)}}|\fT_0(y)-\F_{x_Q,2\de(Q)}(y)||\nabla^2_{xx}K(x,y)|dy\le\\\nonumber
C 
\dn(x)^{-\Nb-2}\int\limits_{{B_{\dn(x_0)+\ve}(x_0)}}|\fT_0(y)-\F_{x_Q,2\de(Q)}(y)|dy,
\end{gather}
 using, on the last step, our estimates for derivatives of  the 
 kernel $K(x,y)$. We apply  estimate \eqref{17} now. For
  $x\in  \overline{B_{\dn(x_0)+\ve}(x_0)},$ the 
difference $\fT_0(y)-\F_{x_Q,2\de(Q)}(y)$ has the form 
 $\F_{x_{Q_1},2\de(Q_1)}(y)-\F_{x_Q,\de(Q)}(y)$, for a certain 
 cube $Q_1$, and  satisfies 
 the conditions of the main theorem with parameters $\de(Q), 
 \de(Q_1)$. Therefore,
\begin{equation}\label{19}
|\fT_0(y)-\F_{x_Q,2\de(Q)}(y)|=|\F_{x_{Q_1},2\de(Q_1)}(y)-\F_{x_Q,\de(Q)}(y)|\le 
c 
\dn(x_0)^{\rb}\om(\dn(x_0)).
\end{equation}
We set $x=x_0$ in \eqref{18.p11}, \eqref{19} and obtain the 
required estimate for 
$\Lc f_0$ outside $\Kb$:
\begin{equation}\label{20}
|(\Lc f_0)(x_0)|\le c \dn^{\rb-2}(x_0)\om(\dn(x_0)).
\end{equation}
 We stress here that by \eqref{20}, the larger $\rb$ in the conditions of the Theorem,  the faster the function $\Lc f_0(x)$ 
 decays as the point $x$ approaches $\Kb$. This fact will be essentially used further on.
   
  \subsection{The integral representation of the function 
  $f_0(x)$.}\label{Integr.repres}
   
  The construction of an approximating functions goes similarly to  
  the one in 
  \cite{RZ25}, with natural modifications.
Let, again, the point $x\in \Ct\Kb$ be fixed, 
$\de_0:=\dist(x,\Kb)>0$. We fix a 
number $\de\in(0,\frac{\de_0}{2})$
 and construct a finite cover of $\Kb$ by balls with radius $\de,$ 
 possessing  Property \ref{MattilaProperty},
 as this was done in  \cite{RZ25}, see Corollary 2.2 there. This 
 implies that for $R_0= \diam(\Kb)$ and $r<R_0$, there exists a 
 collection $\U_\de$
of disjoint  balls $B_{r}(x_{\am})$ centered in $\Kb$, such that the 
concentric balls 
$B_{5r}(x_{\am})$
 form a cover of $\Kb$. Moreover, for any $R\in(\de,R_0)$ and any 
 $\xb\in \Kb$, the 
 quantity of points $x_{\am}$ 
 in the ball $B_{R}(\xb)$ is no greater than $\bb_{\Nb} (R/r)^{\Nb-2}$, 
 with constant 
 $\bb_{\Nb}$ not depending on the radii $r,R$; we apply
 this result  for $r=\frac{\de}{5}$.

We denote by $\Kb_{(\de)}$ the union of balls in $\U_\de$. The 
boundary of 
$\Kb_{(\de)}$ is piecewise smooth: it consists of a finite union of 
parts of spheres with radius $r$.  Let $G^{\circ}(x,y)$ be the 
Green function for 
$\Lc$
 in the domain $\Om^{\circ}$ containing $\Kb$ (recall that 
 $\Om^{\circ}$ is taken to be the unit ball), where 
 the operator $\Lc$ is defined.  Therefore,
  in the integral representation of $f_0$ in the domain 
  $\Om_\de=\Om^\circ\setminus\overline{\Kb_{(\de)}}$, the 
  integrals over $\partial \Om$ 
  vanish,
   and this integral representation takes the form (with some 
   coefficient 
   $c_{\Nb}$)
\begin{gather}\label{23}
  f_0(x)=c_{\Nb}\int_{\partial 
  (\Kb_{\de})}f_0(y)\partial_{\pmb{\pmb{\n}}_{y}}G^{\circ}(x,y)d 
  S(y)\\\nonumber
  -c_{\Nb}\int_{\partial (\Kb_{\de})} 
  \partial_{\pmb{\pmb{\n}}_y}f_0(y)G^{\circ}(x,y)dS(y)
+ c_{\Nb}\int_{\Om_{\de}}\Lc f_0(y)G^{\circ}(x,y)dy, \, 
x\in\Om_{\de}; \end{gather}
here $dS$ denotes the $\Nb$-$1$-dimensional surface measure on the 
piecewise smooth 
surface $\partial (\Kb_{\de})$ and $\pmb{\pmb{\n}}_y$ is the conormal vector associated with operator $\Lc$.

We consider the behavior of each of the terms in \eqref{23} as 
$\de\to 0$, having 
the point $x$ fixed. In the first term, since the distance between
 $x$ and $\partial (\Kb_{\de})$ is separated from zero, the Green 
 function 
 $G^{\circ}(x,y)$ 
 is bounded uniformly in $y$. At the same time, the 
 $\Nb$-$1$-dimensional measure 
 of $\partial \Kb_{\de}$ tends to zero as $\de\to 0$.  To see 
 this, recall that the number of balls in $\U_\de$ is no greater 
 than $C\de^{2-\Nb}$, and the area of the  boundary of each ball 
 is $c\de^{\Nb-1}$.  
  Therefore, the first term in \eqref{23} tends to zero.
  
   To estimate the second term, we let $y$ belong to the closure 
   of some cube $Q$ in the Whitney 
   cover, this means,
    $y\in\partial (\Kb_{\de})\cap \overline{Q}$. Then, according 
    to the definition of $f_0,$
\begin{equation*}%\label{24a}
f_0(y)=\F_{x_Q,2\de(Q)}(y)+\int_{\R^{\Nb}}K(y,w)(\fT_0(w)-\F_{x_Q,2\de(Q)}(w))dw;
\end{equation*}
further on,
\begin{gather}\label{24}
  |\partial_y f_0(y)|\le |\partial_y 
  (f_0(y)-\F_{x_Q,2\de(Q)}(y))|+|\partial_y\F_{x_Q,2\de(Q)}(y)|\le\\\nonumber
  \left|\int_{\R^{\Nb}}\partial_yK(y,w)(\fT_0(w)-\F_{x_Q,2\de(Q)}(w))dw\right|+C,
 \end{gather}
with some absolute constant $C$, since, by \eqref{Est.F}, see 
Lemma \ref{Lem.7.1},  
the functions $\F_{x_Q,2\de(Q)}$ are uniformly bounded.

According to  the definition of the kernel $K$, the integral in 
\eqref{24} is 
spread only over the ball $B_{\dn(y)}(y)$;
for $w$ in this ball, we have, by the conditions of the main 
theorem, 
\begin{gather}\label{25}
  |\fT_0(w)-\F_{x_Q,2\de(Q)}(w)|=|\F_{x_{Q_1},2\de(Q_1)}(w)-\F_{x_Q,2\de(Q)}(w)|\le\\\nonumber 
  c\de(Q)^{\rb}\om(\de(Q))\le C \de^{\rb}\om(\de),
\end{gather}
where $Q_1$ is the cube in the Whitney cover, containing $w$, and 
$x_{Q_1}$ is the point in $\Kb$ closest to its center.
Now we recall the estimates we obtained for derivatives of the 
kernel $K$: 
$|\partial_y K(y,w)|\le c \dn(y)^{-1-\Nb}\le c\de^{-1-\Nb}$;
 therefore, \eqref{24} and \eqref{25} imply
\begin{equation*}%\label{26}
  \left|\int_{\R^{\Nb}}\partial_yK(y,w)(\fT_0(w)-\F_{x_Q,2\de(Q)}(w))dw\right|\le 
  c 
  \de^{-1-\Nb}\de^\rb\om(\de)\de^\Nb\le c\frac{\om(\de)}{\de}.
\end{equation*}
As a result, we obtain
\begin{equation*}%\label{27}
  |\partial_y f_0(y)|\le c \frac{\om(\de)}{\de},
\end{equation*}
(the same estimate that we had in \cite{RZ25}).
Thus, the second term in \eqref{23} tends to zero as $\de\to 0$ 
and we arrive at
 the integral representation for the function $f_0(x)$:
\begin{equation}\label{28a}
  f_0(x)=c_\Nb\int_{\Om\setminus\Kb}\Lc f_0(y) G^{\circ}(x,y)dy.
\end{equation}
Since $\Kb$ has zero Lebesgue measure, we can treat the integral 
in \eqref{28a}, as 
spread over the whole $\Om,$      
      \begin{equation}\label{28}
      f_0(x)=c_\Nb\int_{\Om}\Lc f_0(y) G^{\circ}(x,y)dy, \, 
      x\not\in\Kb. 
      \end{equation}
      % From  estimate 
      %\eqref{20}, 
      %it follows that the right-hand side in \eqref{28} is 
      %continuous on $\Kb$.
      
   Finally, we establish that the integral in \eqref{28} is continuous on $\Kb$ as well. To show this, 
  for a point  $x_0\in \Kb$, we consider  the integral 
 \begin{equation*}%\label{B1}
 F(x_0)=c_\Nb\int_{\Om}\Lc f_0(y)G(x_0,y)dy
 \end{equation*}
    and  estimate $f_0(x)-F(x_0)$, $x\not\in\Kb$, in order to show that 
 $F(x)-F(x_0)\to 0$ as $x\to x_0$. So, for a given $\de,$ we suppose that $|x-x_0|\le 
 \frac{\de}{2}$. We can represent the difference $f_0(x)-F(x_0)$ as
 \begin{gather}\label{B2}
 f_0(x)-F(x_0)=c_\Nb\int_{B_{\de}(x)}\Lc f_0(y)G(x,y)dy-c_\Nb\int_{B_{\de}(x)}\Lc 
 f_0(y)G(x_0,y)dy+\\\nonumber
 \sum_{j=1}^{\infty}\int_{\Ac_j}\Lc f_0(y)(G(x,y)-G(x_0,y))dy\equiv I_0(x)-I_0(x_0)+\sum_{j=1}^{\infty} I_j (x,x_0), 
 \end{gather}
 where $\Ac_j=B_{2^j\de}(x)\setminus B_{2^{j-1}\de(x)}$.
 Using our  estimate for $\Lc f_0(y)$,  we have
 for the terms $I_0(x), \widetilde{I}_0(x_0)$ on the first line in \eqref{B2},
 \begin{gather}\label{B3}
   |\widetilde{I}_0(x_0)|\le C
    \int_{B_{\de}(x)}\frac{d(y)^{\rb-2} \om(d(y))}{|x_0-y|^{\Nb-2}} dy\le\\ \nonumber C 
    \de^{\rb}
     \int_{B_{2\de}(x_0)}\frac{d(y)^{-2} \om(d(y))}{|x_0-y|^{\Nb-2}} dy\le c\om(\de),
 \end{gather}
 according to Lemma 2.4, for $k=0$. Similarly,
 \begin{equation}\label{B4}
   |I_0(x)|\le C \int_{B_\de(x)}\frac{d(y)^{\rb-2}\om(d(y))}{|x-y|^{\Nb-2}}dy\le C\de^\rb 
   \om(\de),
 \end{equation}
 by Lemma 2.2.
 The  term $I_j(x,x_0)$ on the second line in \eqref{B2} is estimated in the following way. For 
 $j\ge1,$ $y\in \Ac_j(x)$,
 we have $|G(x,y)-G(x_0,y)|\le C \de (2^j\de)^{1-\Nb},$ therefore,
 
 \begin{equation*}%\label{B5A}
   |I_j(x,x_0)|\le c \de (2^j\de)^{1-\Nb}\int_{\Ac_j(x)}d(y)^{\rb-2}dy=c\de(2^j)^{\rb-1}\om(2^j\de).
 \end{equation*}
 Using now (2.6), we obtain
 
 \begin{equation*}%\label{B6A}
   |I_j(x,x_0)|\le c \de(2^j\de)^{1-\Nb}(2^j\de)^{\rb+\Nb-2}\om(2^j\de).
 \end{equation*}
 In our case, since $\rb\ge 1$ and the domain $\Om^{\circ}$ is bounded (it is the unit disk), the sum 
 in \eqref{B2} is finite, it contains only terms with $2^{j-1}\de<1$, $j\le N_0$, therefore, 
 $(2^j\de)^{\rb-1}\le C (2^j\de)^{-1}$. This gives for  $I_j(x,x_0)$ the estimate
 
 \begin{equation*}%\label{B7}
 |I_j(x,x_0)|\le C \de (2^j\de)^{-1}\om(2^j\de)=c 2^{-j}\om(2^j\de).
 \end{equation*}
 It follows that
 
 \begin{gather}\label{B8}
   \sum_{j}|I_j(x,x_0)|\le c \sum 2^{-j}\om(2^j\de)\le\\\nonumber
   c\int_0^{\infty}\frac{\om(2^t\de)}{2^t}dt=c'\int_1^\infty\frac{\om(\tau\de)}{\tau^2}d\tau\le 
   c''\om(\de).
 \end{gather}
 Taken together, estimates \eqref{B2}, \eqref{B3}, \eqref{B4}, \eqref{B8} give 
 $|f_0(x)-F(x_0)|<C \om(\de)$. This means that $F_(x)$ converges to $f_0(x_0) $ as $x\to x_0\in 
 \Kb.$

   Since both parts in  \eqref{28}  are continuous on $\Kb$, we see that
  the representation \eqref{28} is valid for all  
      $x\in\Kb$, and therefore on the whole $\Om^{\circ}$.
      
\subsection{Construction of the approximating function 
$v_\de$}\label{Subs.construction}

Before giving a detailed description of the formula \eqref{25} below, we would like to compare this construction with the one 
used in \cite{RZ25} for the case $\rb=0$. In that paper, only the integral term in  \cite{RZ25} was present, and it provided approximation with 
error $\om(\de).$ The extra terms which appear in \eqref{25} are  $\Lc-$harmonic and they improve the quality of approximation when $\rb>0$.

Now we pass to the description of our approximation. We fix a point $\Ob\in\Kb$, which will serve as the starting point 
of our 
construction,
 for all values of the parameter $\de,$ $0<\de\le\diam(\Kb).$ 
 For a given $\de$, we consider the cover $\U_\de,$ as in Property 
 \ref{MattilaProperty},
  by balls $B_{\de}(x_{\am})$.
  We enumerate the points $x_{\am}$ in the following way: the 
  starting numbers go to 
  the points
   $x_{\am}\in \overline{B_{2\de}(\Ob)},$ the following ones go to 
   the 
   points 
   $x_{\am}\in\overline{B_{4\de}(\Ob)}\setminus\overline{B_{2\de}(\Ob)}$, 
   and 
   further on, along the expanding spherical annuli.
    The points, with new numbering, will be denoted 
    $x_\n, \, \n=1,\dots,N.$
     We introduce disjoint sets, $W_1=\overline{B_{2\de}(x_1)}$,
      $W_2=\overline{B_{2\de}(x_2)}\setminus\overline{B_{2\de}(x_1)}$,
       $W_3=\overline{B_{2\de}(x_3)}\setminus(\overline{B_{2\de}(x_2)}\cup\overline{B_{2\de}(x_1)})$,
and so on. If it turns out that for some $\n$, the set $W_\n$ is 
void,  
$W_\n=\varnothing,$ we just delete it and compress the numeration, 
so that, as a result, we have  the sequence of nonempty sets 
$W_\n.$ We define now 
the  sequence of weights $\ro_\n:$ 
\begin{equation}\label{33}
  \ro_\n=(\meas_{\Nb}B_{2\de}(\Ob))^{-1}\int_{W_\n}\Lc f_0(x)dx.
\end{equation}
Definition \eqref{33} and estimates \eqref{20}, \eqref{22} imply

\begin{equation}\label{34}
  |\ro_\n|\le c 
  \de^{-\Nb}\de^{\rb-2+\Nb}\om(\de)=c\de^{\rb-2}\om(\de).
\end{equation}
Next, we define the function
\begin{equation}\label{34a}
  F_\n(x)=c_\Nb\ro_\n\int_{B_{2\de}(V(x_\n,2\de))}G^{\circ}(x,y)dy, 
  \, x\in\Om;
\end{equation}
here $V(x_\n,2\de)$ is the point constructed, with $\Ob= x_0$ in 
place of $x_\n$, in the 
end of Section \ref{SEct.geometry}
 and $c_{\Nb}$ is defined in \eqref{28}. This function is 
 $\Lc$-harmonic outside 
 the ball $B_{2\de}(V(x_\n,2\de)), $ and, therefore, inside $\Kb_{(\de)}$. Finally, for $x\in\Kb_{2\de}$, 
 we define
\begin{equation}\label{35}
  v_{\de}(x)=c_\Nb\int\limits_{\Om\setminus \Kb'_\de}G(x,y)\Lc 
  f_0(y)dy+\sum_{\n=1}^{N} F_\n(x).
\end{equation}
This function is $\Lc$-harmonic in $\Kb_{(\de)}$; it will serve as the required approximation.

\subsection{Estimates for $f_0-v_\de$}\label{Sect.Estim}
We recall that $f_0$ is a smooth extension of the given function 
$f$ from the set 
$\Kb$ to the 
enveloping domain $\Om^{\circ}$ with controlled behavior of 
derivatives and of $\Lc f_0(x)$ as $x$ approaches 
$\Kb$. Thus, on $\Kb$,
 in fact, estimates  for $f_0-v_\de$ coincide with  estimates for  
 $f(x)-v_\de(x)$, 
 this means, they give the quantity we are interested in. 

Using \eqref{28} and \eqref{35}, we can represent the difference 
$f_0(x)-v_\de(x)$, 
$x\in\Kb,$ as 
\begin{gather}\label{36}
 f_0(x)-v_\de(x)= c_\Nb\int_{\Om\setminus\Kb'_\de} 
 G^{\circ}(x,y)\Lc f_0(y)dy 
 +c_\Nb\int_{\Kb'_\de} G^{\circ}(x,y)\Lc f_0(y)dy\\\nonumber 
 -c_\Nb\int_{\Om\setminus\Kb'_\de} G^{\circ}(x,y)\Lc 
 f_0(y)dy-\sum_{\n=1}^{N} 
 F_\n(x)=\\\nonumber
 c_{\Nb}\int_{\Kb'_\de}G^{\circ}(x,y)\Lc f_0(y)dy- \sum_{\n=1}^{N} 
 F_\n(x)=\\\nonumber 
 c_{\Nb}\sum_{\n=1}^N\left(\
\int_{W_\n}  G^{\circ}(x,y)\Lc 
f_0(y)dy-\ro_\n\int_{B_{2\de}(V(x_\n,2\de))}G^{\circ}(x,y)dy\right).
\end{gather}
In the transformation in \eqref{36}, we used the fact that the 
sets $W_\n$ are 
disjoint, $W_\n\subset \overline{B_{2\de}(x_\n)}$, and their union 
is $\Kb'_{\de}$.
 We choose the number $M$ so that 
\begin{equation*}
  \diam(\Kb)<2^M\de\le 2\diam(\Kb).
\end{equation*}
With this notation, the last expression in \eqref{36} can be 
transformed in the 
following way:
\begin{gather}\label{37}
\sum_{\n=1}^N\left(\
\int_{W_\n} G^{\circ}(x,y)\Lc 
f_0(y)dy-\ro_\n\int_{B_{2\de}(V(x_\n,2\de))}G^{\circ}(x,y)dy\right)=\\\nonumber
 \sum\limits_{\n: x_\n\in B_{4\de}(x)}\left(\
\int_{W_\n} G^{\circ}(x,y)\Lc 
f_0(y)dy-\ro_\n\int_{B_{2\de}(V(x_\n,2\de))}G^{\circ}(x,y)dy\right)=\\\nonumber
  {\sum_{l=3}^M}{\sum_{\n}}^{(l)}\left(\
\int_{W_\n} G^{\circ}(X,y)\Lc 
f_0(y)dy-\ro_\n\int_{B_{2\de}(V(x_\n,2\de))}G^{\circ}(x,y)dy\right)\\\nonumber
\overset{def}{=} \Ic_0+{\sum^{M}_{l=3}}{}^{{}^{(l)}} \Ic_l,
\end{gather}
where the superscript $(l)$ in  $\sum^{(l)}$ indicates the 
summation over those $\n$ for which $ x_\n\in 
\overline{B_{2^l\de}(x)}\setminus{B_{2^{l-1}\de}(x)}$.
The term $\Ic_0$ contains no more than $4^{\Nb-2}b_{\Nb}$ 
summands.  For each of 
them, this means, for $x_\n\in B_{4\de}(x)$, we use estimate 
\eqref{20}, 
the general estimate $|G(x,y)|\le c|x-y|^{2-\Nb}$, and the 
estimate of the integral 
\eqref{21}. This gives us
\begin{equation}\label{38}
  \left|c_{\Nb}\int_{W_\n} G^{\circ}(x,y) \Lc f_0(y)dy\right|\le 
  c\int\limits_{{B_{2\de}(x_\n)}}\dn(y)^{\rb-2}\om(\dn(y))|x-y|^{2-\Nb}dy\le 
  C\de^{\rb}\om(\de).
\end{equation}Next, for $y\in B_{2\de}(V(x_\n,2\de)),$ we have 
$|G^{\circ}(x,y)|\le c 
|x-y|^{2-\Nb}\le c \de^{2-\Nb}$, therefore, \eqref{34} implies
\begin{equation}\label{39}
  \left|c_{\Nb}\ro_\n\int_{B_{2\de}(
  x_\n)}G^{\circ}(x,y)dy\right|\le 
  c\de^{\rb-2}\om(\de)\times 
  \de^{2-\Nb}\de^{\Nb}=c\de^{\rb}\om(\de).
\end{equation}
Estimates \eqref{38}, \eqref{39} produce the bound for $\Ic_0$:
\begin{equation*}%\label{40}
|\Ic_0|\le c \de^{\rb}\om(\de).
\end{equation*}

Next, we consider the term $\Ic_l,$ $l\ge 3$ in \eqref{37}, this means, the sum 
over such $\n$ for which 
$2^{l-1}\de\le|x_\n-x|\le 2^l\de.$ There are no more
 than $2^{l(\Nb-2)}b_{\Nb}$ points $x_\n$ in this spherical 
 annulus. Considering  one of 
 these points,
 we choose  arbitrarily two additional  points $y_{\n 1}\in W_\n$ 
 and $y_{\n 2}\in 
 \overline{B_{2\de}(V(x_\n,2\de))}.$ Then we have
\begin{gather}\label{41}
  \int_{W_\n}G^{\circ}(x,y)\Lc f_0(y)dy -
  \ro_\n\int\limits_{{B_{2\de}(V(x_\n,2\de))}}G^{\circ}(x,y)dy=\\\nonumber
  \int_{W_\n}\Lc f_0(y)\left(G^{\circ}(x,y)-G^{\circ}(x,y_{\n 
  1})\right)dy+G^{\circ}(x,y_{\n 
  1})\int_{W_\n}\Lc f_0(y)dy-\\\nonumber
\ro_\n\!\!\!\!\!\!\!\int\limits_{{B_{2\de}(V(x_\n,2\de))}}(G^{\circ}(x,y)-G^{\circ}(x,y_{\n 
2}))dy-\ro_\n 
G^{\circ}(x,y_{\n 
2})\int\limits_{{{B_{2\de}(V(x_\n,2\de))}}}dy=\\\nonumber
\int_{W_\n}\Lc f_0(y)\left(G^{\circ}(x,y)-G^{\circ}(x,y_{\n 
1})\right)dy 
-\ro_\n\int\limits_{{B_{2\de}(V(x_\n,2\de))}}(G^{\circ}(x,y)-G^{\circ}(x,y_{\n 
2}))dy+\\\nonumber
\left(G^{\circ}(x,y_{\n 1})-G^{\circ}(x,y_{\n 
2})\right)\int_{W_\n}\Lc f_0(y)dy.
\end{gather}
When performing transformations in \eqref{41}, we used the 
definition of the 
coefficient $\ro_\n$ in \eqref{33}. 

We pass to estimating separate terms in \eqref{41}. For $y\in 
\overline{B_{2^l\de}(x)}\setminus B_{2^{l-1}\de}(x)$, 
we have $|\nabla_{y}G^{\circ}(x,y)|\le |x-y|^{1-\Nb}\le C 
(2^l\de)^{1-\Nb}$. Therefore,
\begin{gather}\label{42}
|G^{\circ}(x,y)-G^{\circ}(x,y_{\n 1})|=\left|\int_0^1 
D_{t}(G^{\circ}(x, y+t(y_{\n 1}-y)))dt\right|\le\\\nonumber c 
\de|x-y|^{1-\Nb}\le c \de (2^l\de)^{1-\Nb}.
\end{gather}
In a similar way, using the estimate for the derivatives of the 
Green function and 
taking
 into account the position of the points $y,y_{\n 1}, y_{\n 2}$, 
 we obtain
\begin{equation}\label{43}
  \left| G^{\circ}(x,y)-G^{\circ}(x,y_{\n 2})\right|\leq \de 
  (2^l\de)^{1-\Nb},
\end{equation}
\begin{equation}\label{44}
   \left| G^{\circ}(x,y_{\n 1})-G^{\circ}(x,y_{\n 2})\right|\leq 
   \de (2^l\de)^{1-\Nb}.
\end{equation}

Adding up inequalities \eqref{41}-\eqref{44}, we arrive at the  
estimate for a 
single term in $\Ic_{l}$:
\begin{gather}\label{45}
  c_{\Nb}\left| \int_{W_\n}G^{\circ}(x,y)\Lc 
  f_0(y)dy-\ro_\n\int\limits_{{B_{2\de}(V(x_\n,2\de))}}G^{\circ}(x,y)dy\right|\le 
  \\\nonumber
 c\de(2^l\de)^{1-\Nb}\left(\int_{W_\n}|\Lc 
 f_0(y)|dy+\ro_\n\meas_{\Nb}B_{2\de}(\Ob)\right)\le 
  \\\nonumber 
  c\de(2^l\de)^{1-\Nb}\left(\int_{W_\n}\dn(y)^{\rb-2}\om(\dn(y))dy+\de^{\rb-2+\Nb}\om(\de)\right)\le
c\de^{\rb-2+\Nb}\om(\de)\de(2^l\de)^{1-\Nb};
\end{gather}
in transformations in \eqref{45}, we used \eqref{34} and 
\eqref{22}. Since in the 
spherical annulus
 $B_{2^l\de}(x)\setminus B_{2^{l-1}\de}(x),$ there are no more 
 than $c 
 2^{l(\Nb-2)}$
 points $x_\n$, we obtain the required estimate for $\Ic_l$:
\begin{equation*}%\label{46}
  |\Ic_l|\le c 2^{l(\Nb-2)}\de^\rb\om(\de)\cdot 
  \frac{1}{2^{l(\Nb-1)}}=C 
  2^{-l}\de^\rb\om(\de).
\end{equation*}
Now we sum over $l$ and  arrive at
\begin{equation*}%\label{47}
  |f_0(x)-v_\de(x)|\le c\de^{\rb}\om(\de)(1+\sum_{l\ge3} 
  2^{-l})=C\de^{\rb}\om(\de).
\end{equation*}

This inequality proves the first statement of the main theorem. 
The second part 
follows easily from the first one since:
\begin{equation}\label{48}
|v_{\de}(x)-v_{\frac{\de}{2}}(x)|\le 
|v_{\de}(x)-f_0(x)|+|v_{\frac{\de}{2}}(x)-f_0(x)|\le 
c\de^{\rb}\om(\de).
\end{equation}
Thus, the 'only if' part of the Theorem is proved for 
$x\in\Kb'_{\de}\supset\Kb_\de.$

The 'if' statement follows by setting $\F_{x,\de}(y)=v_{\de}(y),\,  
y\in 
B_{\de}(x),$ for all $x\in\Kb$, this means, we take the single 
function $v_{\de}(y)$ as the local approximates for $f$ at all points $x\in\Kb$. 
The required property of the function $\F_{x,\de}$
 follows from \eqref{48} and the  equality
\begin{equation*}
  v_\de-v_{2^{-3}\de}=\sum_{s=0}^{2}(v_{2^{-s}\de}-v_{2^{-s-1}\de}).
\end{equation*}

\section{Generalized  derivatives of $f(x), \, 
x\in \Kb$}\label{SEct.derivatives}
In this section, we \emph{define} generalized derivatives of the 
function 
$f_0(x)$ at 
points $x\in\Kb$, namely, points, where the usual derivatives, 
generally, do not exist. 
These derivatives are used to define surrogates of derivatives of the 
initial function 
$f$. We show here that this definition is self-consistent, and 
then we prove that the 
derivatives of the approximating functions $v_\de$ converge to 
these generalized 
derivatives of $f$ on $\Kb$ as $\de\to 0$.  Naturally, the higher 
derivatives we consider, the more
 smoothness we require from the coefficients of the operator 
 $\Lc$.
\subsection{Definitions}\label{Def.gener.}
Let $\a$ be a multi-index, $1\le |\a|=k\le \rb,$ $c_{\Nb}$ is the 
constant in 
\eqref{28}, where the representation for $f_0$ is found.
\begin{defin}For $x\in\Kb$, we define the generalized derivative 
$f_{(\a)}(x)$ by
\begin{equation}\label{A10}
  f_{(\a)}(x):= c_{\Nb}\int_{\Om^{\circ}}\Lc f_0(y)\partial_{x}^\a 
  G^{\circ}(x,y)dy+\sum_{\n}\partial^{\a}F_{\n}(x),
\end{equation}
where $G^{\circ}(x,y)$ is the Green function for $\Lc$ in 
$\Om^{\circ}$.
\end{defin}
This means that we define derivatives of $f_0$ by, still formal, 
differentiation of 
the representation \eqref{28}. 

To justify the definition, we need first to prove that the 
integral in \eqref{A10} converges.
In fact, for the function $f_0(x),$ as defined earlier, we have 
the estimate
\begin{equation*}
  |\Lc f_0(y)|\le c d(y)^{\rb-2}\om(d(y)).
\end{equation*}
Suppose that the coefficients of the operator $\Lc$ belong to 
$C^{3+|\a|}$. Since, 
by \eqref{Kras.est},
\begin{equation*}
  |\partial_x^{\a}G^{\circ}(x,y)|\le c |x-y|^{-(\Nb-2+k)}, \, k=|a,|
\end{equation*}
we have
\begin{equation*}
  |\Lc f_0(y)\partial^{\a}G^{\circ}(x,y)|\le C 
  \frac{d(y)^{\rb-2}\om(d(y))}{|x-y|^{\Nb-2+k}},
\end{equation*}
and, since $k=|\a|\le \rb,$ we can use the results of Lemma \ref{lemA1},
\ref{lem.int.4}, 
therefore, the integral in \eqref{A10} converges.

\subsection{Proof of Theorem \ref{Thm.quality}}  In this section,  
$\|\cdot\|_{\Kb_{\frac{\de}{2}}}$ denotes the norm in 
$C(\Kb_{\frac{\de}{2}})$ etc.
We prove the estimate \eqref{appr.deriv} first.
For some $\D>0$, fix a point $x^0\in \Kb_{\D/2}.$ Then, for $x\in \Kb_{\D},$ we have $\Lc v_{\D}(x)=\Lc v_{2\D}(x)=0.$
Consider the Green function  $G_\D(x,y)$ for $\Lc$ in the ball $B_{\D}\equiv B_{\D}(x^0).$ Then
\begin{equation*}
  v_{\D}(x)-v_{2\D}(x)=c_{\Nb}\int_{\partial B_{\D}}(v_{\D}(y)-v_{2\D}(y))\partial_{\pmb{\n}_y}G_\D(x,y) dS(y).
\end{equation*}
We differentiate this equality  $\rb+1$ times in $x$:

\begin{equation*}
  \nabla^{\rb+1}(v_{\D}(x)- v_{2\D}(x))=c_{\Nb}\int_{\partial B_{\D}}(v_{\D}(y)-v_{2\D}(y))\partial_{\pmb{\n}_y} \nabla_x^{\rb+1}G_\D(x,y) dS(y).
\end{equation*}
Now we use the estimate for the derivative of the Green function and obtain, for $x\in B_{\D/2},$ see Corollary \ref{cor.Kras}:
\begin{gather}\label{Additional}
  |\nabla^{\rb+1}(v_{\D}(x)- v_{2\D}(x))|\le C \D^{\rb}\om(\D)\int_{\partial B_{\D}}|x-y|^{-\Nb+1-\rb}dS(y)\le \\\nonumber C   \D^{-1}\om(\D).
 \end{gather}

With $\de$ fixed, we take an integer $M$ so that $\diam(\Kb)< 
2^M\de\le 
2\diam(\Kb),$ and we write \eqref{Additional} for $\D=2^{j-1}\de,$ $1\le j\le M$. Adding the corresponding estimates, we obtain and  
\begin{gather}\label{AA.1}
  \|\nabla^{\rb+1}v_{\de}\|_{\Kb_{\frac{\de}{2}}} \le 
  \sum_{j=1}^{M}\|\nabla^{\rb+1}(v_{2^{j-1}\de}-v_{2^{j}\de})\|_{\Kb_{\frac{\de}{2}}}+\\\nonumber
\|\nabla^{\rb+1}(v_{{}_{2^M\de}})\|_{\Kb_{\frac{\de}{2}}}\le 
\|\nabla^{\rb+1}(v_{2^M\de})\|_{\Kb_{\frac{\de}{2}}}+O(1).
\end{gather}
due to the third property of the approximating function in 
\eqref{appr.function} and Lemma \ref{Lem.7.1}.

After this, the sum $\sum_{j=1}^{M}\de^{-1} 2^{-j}\om(2^j\de)$ is 
estimated similarly 
to the sums in Sect. \ref{Sect.Integrals}, via the integral
\begin{equation*}%\label{AA.3}
  \de^{-1}\int_{1}^M 2^{-\tau}\om(2^{\tau}\de)d\tau\le 
  c\int_{2\de}^{\infty}\frac{\om(s)}{s^2}ds\le c 
  \frac{\om(\de)}{\de},
\end{equation*}
and this proves the inequality.

Now we prove the approximation property.
It follows from \eqref{35} that 
\begin{equation}\label{A15}
  \partial^{\a}v_\de(x)=c_\Nb \int_{\Om\setminus \Kb_{2,\de}}\Lc 
  f_0(y)\partial_x^{\a}G^{\circ}(x,y)dy+
\sum_{\n=1}^{N}\partial^{\a}F_\n(x), \, x\in\Kb, 
\end{equation}
where $F_\n (x)$ are functions constructed for the given $\de$ as 
in Sect. 
\ref{Subs.construction}.

Next, from the definition \eqref{A10}, and \eqref{A15}, an 
estimate for the 
approximation of derivatives follows,
\begin{gather}\label{A16}
  f_{(\a)}(x)-\partial^{\a}v_{\de}(x)=c_{\Nb}\int_{\Kb_{2,\de}}\Lc 
  f_0(y)\partial^{\a}_xG^{\circ}(x,y) dy 
  -\sum_{\n=1}^{N}\partial^{\a}F_\n(x)=\\\nonumber
  \sum_{\n=1}^{N}c_{\Nb}\int_{W_\n}\Lc 
  f_0(y)\partial_{x}^{\a}G^{\circ}(x,y)dy-\sum_{\n=1}^{N}\partial^{\a}F_\n(x).
\end{gather}
Now we can use again  the estimates of $\Lc f$ obtained in Sect. 
\ref{Sect.estimat.f0(x)} and of derivatives of the Green function, 
which gives
\begin{equation*}
  \left|\int_{W_\n}\Lc f_0(y)\partial_x^\a G^{\circ}(x,y) 
  dy\right|\le c 
  \int_{B_{2\de}(x_\n)}\frac{d(y)^{\rb-2}\om(d(y))}{|x-y|^{\Nb-2+|\a|}}dy.  
\end{equation*}
If the point $x_\n$ is close to $x$, namely,  $x_\n\in 
B_{4\de}(x),$ we can use 
estimates \eqref{A5} and \eqref{38}, which give
\begin{equation*}%\label{A17}
  \int_{B_{2\de}(x_\n)}\frac{d(y)^{\rb-2}\om(d(y))}{|x-y|^{\Nb-2+k}}dy 
  \le c 
  \de^{\rb-k}\om(\de),
\end{equation*}
since here $B_{2\de(x_\n)}\subset B_{6\de(x)}$. 

For the derivative $\partial^\a F_{\n}$ the estimate in this 
domain is easier. It 
follows from \eqref{34} that
\begin{equation}\label{F11}
  \partial^{\a}F_{\n}(x)=c_\Nb\ro_\n\int_{B_{2\de}(V_{x_\n,\de 
  })}\partial_x^{\a}G(x,y)dy.
\end{equation}
By our estimate \eqref{34a}, 
\begin{equation}\label{F12}
  |\ro_\n|\le c \de^{\rb-2}\om(\de).
\end{equation}
For $x_{\n}\in \overline{B_{4\de}(x)},$ relations \eqref{F11}, 
\eqref{F12} imply
\begin{gather}\label{F3}
  \left|\partial^\a F_{\n}(x)\right|\le c 
  \de^{\rb-2}\om(\de)\int_{B_{2\de}(V_{x_\n,2\de })}|\partial_x^\a 
  G^{\circ}(x,y)| dy\le\\\nonumber c 
  \de^{\rb-2}\om(\de)\int_{B_{2\de}(V_{x_\n,2\de})}|x-y|^{-\Nb+2-|\a|} 
  dy.
\end{gather}
Since $|V(x_{\n,2\de})-x_\n)|\ge 12\de$, we have $|y-x|\ge 
12\de-2\de-4\de=6\de$ for 
$y\in{B_{2\de}(V_{x_\n , 2\de})}$ and $x_{\n}\in B_{4\de}(x)$, 
therefore, we obtain 
from \eqref{F3}:
\begin{equation*}%\label{F4}
 |\partial^\a F_{\n}(x)|\le c \de^{\rb-2}\om(\de)\cdot 
 \de^{\Nb}\de^{-\Nb+2-|\a|}=c \de^{\rb-|\a|}\om(\de). 
\end{equation*}

We consider now those $x_\n$ which lie outside the ball 
$B_{4\de}(x)$,
 i.e.,  $x_\n\in B_{2^{l+1}\de}(x)\setminus B_{2^{l}\de}(x)$  for 
 some $l\ge 2$. In 
 this case, for $y\in B_{2\de}(x_{\n})$, we have $|x-y|\asymp 2^l 
 \de$, and, again, 
 using Green functions estimates \eqref{Kras.est} we have
\begin{gather}\label{A18}
 |S_{\n}|:= \left| \int_{W_\n}\Lc f_0(y)\partial_x^\a 
 G(x,y)dy\right|\le C 
 \int_{B_{2\de}(x_\n)}\frac{d(y)^{\rb-2}\om(d(y))}{|x-y|^{\rb-2+k}}dy\le 
 \\\nonumber
  c\frac{\de^{\Nb-2+\rb}\om(\de)}{(2^{l}\de)^{\Nb-2+k}}=c\de^{\rb-k}\om(\de)2^{-(\Nb-2+k)j},\, 
  k=|\a|.
\end{gather}

We recall that in the proof of the estimate for $f(x)-v_\de(x)$ we 
used the fact 
that the spherical annulus $B_{2^{l+1}\de}(x)\setminus 
B_{2^{l}\de}(x)$ contains no 
more than $c 2^{(\Nb-2)l}$ points $x_{\n}$, therefore, inequality 
\eqref{A18} 
implies  the estimate

\begin{gather}\label{A19}
 \sum_{\n:x_\n\in  B_{2^{l+1}\de}(x)\setminus 
 B_{2^{l}\de}(x)}|S_{\n}|\le   
 \\\nonumber c 
 2^{(\Nb-2)l}\de^{\rb-k}\om(\de)\frac{1}{2^{(\Nb-2+k)l}}
  =c 2^{-lk}\de^{\rb-k}\om(\de).
\end{gather}

To obtain the  estimate for $\sum_{\n=1}^\infty\partial^{\a}F_\n$ 
a 
similar, but much simpler calculation works.
We collect estimates \eqref{A16}-\eqref{A19} to obtain the 
required inequality
\begin{equation*}%\label{A20}
  |f_{(\a)}(x)-\partial^{\a}v_{\de}(x)|\le c 
  \de^{\rb-k}\om(\de)\sum_{l=1}^{\infty}2^{-l}+c\de^{\rb-k}\om(\de)=c\de^{\rb-k}\om(\de).
\end{equation*}

\subsection{Taylor remainder estimates} We have defined, for a function 
$f\in\Hc_{\Lc}^{\rb+\om}(\Kb)$ admitting 
local approximation by $\Lc-$harmonic functions, certain 
surrogates of derivatives.  The 
same kind of calculations as in the previous subsection, which 
established the 
convergence of derivatives of the  approximating functions $v_\de$ 
to the 
generalized derivatives of $f$, enables one to prove that in a 
certain sense, these 
generalized derivatives may be used to construct a Taylor type 
formula for $f$ and 
its derivatives. We give here  only the formulation.

\begin{thm}\label{Thm Holder estim}For the function $f$ and the 
compact set $\Kb$ 
satisfying the conditions of this paper, provided the coefficients of $\Lc$ belong to $C^{\rb+4}$,
the following inequality holds, with some constant $c$ not 
depending on $x_1,x_2\in 
\Kb$: 
\begin{gather*}%\label{A12}
  \left|f_{(\a)}(x_2)-f_{(\a)}(x_1)-\sum_{|\be|\ge 1, 
  |\a|+|\be|\le 
  \rb}(\be!)^{-1}f_{(\a+\be)}(x_1)(x_2-x_1)^{\be}\right|\le\\\nonumber 
  c 
  |x_2-x_1|^{\rb-k}\om(|x_2-x_1|),\, 
\end{gather*}
for $1\le |\a|=k<\rb.$ As a limit case, 
\begin{equation*}
  |f_{(\a)}(x_1)-f_{(\a)}(x_2)|\le c \om(|x_1-x_2|), \, |\a|=\rb.
\end{equation*}
\end{thm}
The, rather technical, proof, is based upon the Taylor expansion 
of the Green 
function substituted into the expression for $f_0$.

\section{The counter-example} In this counter-example, we show 
that if we relax 
the local approximation condition imposed on the
 function $f$ defined on the compact set $\Kb$ in the main 
 theorem, namely, if we only suppose 
 that the approximating functions $\F_{x,\de}$ are just smooth, 
 without
  requiring that they are $\Lc$-harmonic, then 
  the global approximation by
   solutions of this equation may fail.  This effect is surely 
   caused by
   a pathological structure of the set $\Kb$. We note that such 
   counter-examples 
   are possible
    only in the case $\rb\ge1$: in our paper \cite{RZ25}, we have 
    shown that
    for a minor smoothness, i.e., $\rb=0,$  this means, for the 
    approximation with 
    quality $\om(\de),$
     the requirement that the locally approximating 
    functions are $\Lc-$harmonic does not arise.
    
     The compact set $\Kb$ in our example looks as a dish-brush with dense $\Nb$-$2$- dimensional bristles looking in different directions.
     So, for a harmonic function $v_{\de}$ approximating on $\Kb$ the given function $f$, 
    all second derivatives of $v_{\de}$ should approximate  all second derivatives of $f,$ therefore, if $f$ is a trace 
    on $\Kb$ of a \emph{non-harmonic} function, such approximation is impossible. On the opposite, if $\Kb$ were more 
    regular, say a $C^{2}$-surface of codimension $2$, the global approximation by  harmonic functions $v_{\de}$ would impose restriction only upon 
    \emph{some} partial derivatives of $f$,  not causing a contradiction.
    
    Now we render concrete the above description. 

We introduce here a special notation for $\Nb$-$2$-dimensional 
balls in $\R^{\Nb}$: 
this notation will reflect the orientation
of these balls in $\R^{\Nb}.$ We set
\begin{gather*}
  B_{\pk 1}^*= \{x=(x_1,\dots, x_{\Nb-1}, x_{\Nb})\}: |x|\le 
  2^{-\pk-2}, \, 
  x_{\Nb-1}=x_{\Nb}=0; \\\nonumber
 B_{\pk 2}^*= \{x=(x_1,\dots, x_{\Nb-1}, x_{\Nb})\}: |x|\le 
 2^{-\pk-2}, \, 
 x_{\Nb-2}=x_{\Nb-1}=0; \\\nonumber
 B_{\pk 3}^*= \{x=(x_1,\dots, x_{\Nb-1}, x_{\Nb})\}: |x|\le 
 2^{-\pk-2}, \, 
 x_{\Nb-2}=x_{\Nb}=0, \\\nonumber
 B_{\pk \io}=B_{\pk \io}^*+\left(\frac12(2^{-\pk-1}+2^{-\pk}), 
 0,\dots,0\right),\,\pk=0,1,\dots,\, \io=1,2,3.
\end{gather*}
The compact set $\Kb\subset \R^{\Nb}$ is defined as
\begin{equation*}%\label{48a}
  \Kb=\left(\bigcup_{\io=1}^{3}\bigcup_{\pk=0}^\infty 
  B_{\pk\io}\right)\bigcup 
  \{\mathbb{0}_{\Nb}\}, 
\end{equation*}
where $\{\mathbb{0}_{\Nb}\}$ denotes the origin in $\R^\Nb$.
It is easy to check that this set is  Ahlfors-David-regular of 
dimension  $\Nb$-$2$.

We consider the function
  $f(x)=|x|^2,$ $x\in \Kb$. The same function, considered in 
  $\R^{\Nb}$, serves as 
  a local smooth approximation
  $\F_{\de,x}(y)$ for itself in any neighborhood of any point in 
  $\Kb$, for any level 
  of smoothness, since
    $\F_{\de,x}(y)-f(y)\equiv 0.$ The only shortcoming, compared 
    with the 
    conditions of Theorem \ref{MainTheorem},
     is that the approximating function is not a solution of the 
     Laplace equation. 

\begin{thm}\label{counterex} It is impossible to approximate 
$f(x)$ in the sense 
of Theorem \ref{MainTheorem}  with $\rb=2,$  and 
$\om(\de)=\de^{\s},$ $0<\s<1$, by 
harmonic functions, this means,
 by solutions of the equation $\Lc v_\de\equiv-\D v_\de=0$,
\end{thm}
In other words, for such a wild set $\Kb,$ one cannot approximate on $\Kb$ the non-harmonic function $f(x)$ by harmonic functions, even locally.  
\begin{proof}
Suppose that the approximation in question is possible, thus, for 
any 
$\de\in(0,1),$ there exists a  function
$v_\de$ such that, 
\begin{equation}\label{49}
  |v_\de(x)-f(x)|\le \mathbbm{c}\de^{2+\s}, \, x\in\Kb,
\end{equation}
\begin{equation}\label{50}
  |v_\de(x)-v_{\frac{\de}{2}}(x)|\le \mathbbm{c}\de^{2+\s}, \, 
  x\in\Kb_{\frac{\de}{2}},
\end{equation}
with some $\mathbbm{c}$ not depending on $\de,$ and
\begin{equation}\label{51}
  \D v_{\de}(x)=0, \, x\in\Kb_{\de}.
\end{equation}
 We establish the following property.
 \begin{lem}\label{lem.counter} Under the assumptions 
 \eqref{49}-\eqref{51}, the 
 function $v_{\de}$ must
 satisfy the  estimate: 
 \begin{gather}\label{52}
 |\nabla^3 v_{2\de_{\pk}}(x)|\le c 
 \mathbbm{c}\de_{\pk}^{\s-1},\,\\\nonumber 
 x\in U_{\pk}:=(\tb_{\pk}+B_{\de_{\pk}}(\mathbb{0}), \, 
 \tb_{\pk}=(\frac{1}{2}(2^{-\pk-1}+2^{-\pk}),0,\dots,0), \, 
 \de_{\pk}=2^{-\pk-2}.
 \end{gather}
 \end{lem}
 \begin{proof}
 To prove \eqref{52}, we  denote by  $\f_k(x), \, x\in U_{\pk}$, 
 the function 
 $\f_k(x)=v_{2^k\de_m}(x)-v_{2^{k+1}\de_m}(x)$,\,
  $k=1,\dots,N$, where $N$ is chosen so that $1<2^N\de\le 2$. 
  Using this function, 
  we can represent $v_{2\de_{\pk}}$ and its order 3 gradient as
  
  \begin{equation}\label{53}
    v_{2\de_{\pk}}=v_{2^{N+1}\de}+\sum_{k=1}^N\f_k,\, \mbox{and}\,  
    \nabla^3v_{2\de_{\pk}}=\nabla^3v_{2^{M+1}\de}+\sum_{k=1}^M\nabla^3\f_k.
  \end{equation}
  Due to the definition of $N$, we have 
  $|\nabla^3v_{2^{N+1}\de}|\le c$. For the functions 
  $\f_k$, we have the estimate, by the assumption
   \eqref{50}:
 
   \begin{equation*}%\label{54}
   |\f_k(x)|\le c \mathbbm{c} (2^k\de_{\pk})^{2+\s}, 
   x\in\Kb_{2^k\de_{\pk}}.
   \end{equation*}
   Derivatives of the function $\f_k$ which is  harmonic in the 
   ball 
   $U_{k,\pk}=\tb_{\pk}+B_{2^k\de_{\pk}}(\mathbb{0})\subset\Kb_{2^k\de_{\pk}}$, 
   can be estimated 
   using the Poisson formula:
   
   \begin{equation}\label{55}
     |\nabla^3\f_k(x)|\le c 
     \mathbbm{c}\frac{(2^k\de_{\pk})^{2+\s}}{(2^k\de_{\pk})^3}=c\mathbbm{c} 
     2^{k(\s-1)}\de_{\pk}^{\s-1}.
   \end{equation}
   Now, it follows from \eqref{53}, \eqref{55} that 
   
   \begin{equation*}
     |\nabla^3 v_{2\de_{\pk}}(x)|\le 
     c\mathbbm{c}\de_{\pk}^{\s-1}\sum_{k=1}^M 
     2^{k(\s-1)}+c\le c'\de_{\pk}^{\s-1},
   \end{equation*}
   and this proves Lemma \ref{lem.counter}.
   \end{proof}
   Having this estimate, we proceed with our example. We introduce 
   the function
 
   \begin{equation}\label{56}
     P_2(x,\tb_{\pk})=v_{2\de_{\pk}}(\tb_{\pk})+\sum_{|\a|=1}\frac{1}{\a!}\partial_{\a}v_{2\de_{\pk}}(\tb_{\pk})(x-\tb_{\pk})^{\a}+
     \sum_{|\a|=2} 
     \frac{1}{\a!}\partial_{\a}v_{2\de_{\pk}}(\tb_{\pk})(x-\tb_{\pk})^{\a};  
     \end{equation}
this is the second degree Taylor polynomial for 
$v_{2\de_{\pk}}(x)$  at the point 
$x=\tb_{\pk}$.
We use the integral form of the remainder term in the Taylor 
formula, to 
 express the difference of values  at the points $x$ and 
 $\tb_{\pk}$:
 which, according to \eqref{56}, gives
 
 \begin{equation*}%\label{57}
   |\nabla^2v_{2\de_{\pk}}(x)-\nabla^2 P_2(x,\tb_{\pk})|\le c 
   \de_{\pk}\sup_{y\in 
   \tb_{\pk}+B_{\de_{\pk}}(\mathbb{0}_{\Nb})}|\nabla^3 
   v_{2\de_{\pk}}(y)|\le c\mathbbm{c}\de_{\pk}\cdot 
   \de_{\pk}^{\s-1}=c\mathbbm{c}\de_{\pk}^{\s}.
 \end{equation*}
 In particular, this implies 
 
 \begin{equation}\label{58}
   |\D v_{2\de_{\pk}}(x)-\D P_2(x,\tb_{\pk})|\le 
   c\mathbbm{c}\de_{\pk}^{\s}, \, x\in 
   \tb_{\pk}+B_{\de_{\pk}}(\mathbb{0}_{\Nb}).
 \end{equation}
 Combining \eqref{51} and \eqref{58}, we obtain
 
 \begin{equation}\label{59}
   |\D P_2(x,\tb_{\pk})|\le c\mathbbm{c}\de_m^{\s}, \, x\in 
   \tb_{\pk}+B_{\de_{\pk}}(\mathbb{0}_{\Nb}).
 \end{equation}
 Next,  we represent $P_2(x,\tb_{\pk})$ in a different form, with 
 some terms, second 
 order homogeneous, separated:
 
 \begin{gather}\label{60}
 P_2(x,\tb_{\pk})= \sum_{j=1}^\Nb 
 b_{j}(\tb_{\pk},\de_{\pk})(x_j-\tb_{\pk,j})^2+\tilde{P}_2(x,\tb_{\pk}),\,\\\nonumber
  \tb_{\pk}=(\frac{1}{2}(2^{-\pk}+2^{-\pk-1}),0,\dots,0),
 \end{gather}
 where the polynomial $\tilde{P}_2(x,\tb_{\pk})$ contains terms of degree 
 0 and 1 in $x-\tb_{\pk}$ 
 as well as mixed
   terms of the form $c (x_j -\tb_{\pk,j})(x_j -\tb_{\pk,j'})$ ,
 $j\ne j'$.
 Since $\D \tilde{P}_2=0,$  \eqref{59}, \eqref{60} imply 
 
 \begin{equation}\label{61}
 |\sum_{j=1}^\Nb b_{j}(\tb_{\pk},\de_{\pk})|\le 
 c\mathbbm{c}\de_{\pk}^{\s}.
 \end{equation}
 We recall that the coefficients $b_{j}(\tb_{\pk},\de_{\pk})$ in 
 \eqref{61} are constant
 times the second derivatives of $v_{2\de_{\pk}}(\tb_{\pk}),$
 $b_{j}(\tb_{\pk},\de_{\pk})=\frac12 (\partial^2_{j,j} 
 v_{2\de_{\pk}})(\tb_{\pk}).$
 The set 
 $\Kb\cap(\tb_{\pk}+\mathbf{{B_{\de_{\pk}}(\mathbb{0}_{\Nb})}})$ 
 contains the 
 closed intervals
 
 \begin{gather*}
   \Ib_{1,\pk}=[2^{-\pk-1},2^{-\pk} ]\times  
   \{\mathbb(0)_{\Nb-1}\},\,\mbox{and}\, 
   \\\nonumber
   \Ib_{j,\pk}=\{\frac{1}{2}(2^{-\pk}+2^{-\pk-1}), 0,\dots, 
   \overset{j}{ 
   \overbrace{[-\de_{\pk},\de_{\pk}]}},0,\dots,0.\},\, 
   \mbox{for}\, j>1 ,
 \end{gather*}
 where the overset $\overset{j}{\dots}$ denotes the component at 
 the $j$-th 
 co-ordinate.
 
 Now, it follows from \eqref{49} that 
 
 \begin{equation*}%\label{62}
   |v_{2\de_{\pk}}(x)-|x|^2|\le c\mathbbm{c}\de_{\pk}^{2+\s}, \, 
   x\in\bigcup_{j=1}^\Nb 
   \Ib_{j,\pk}.
 \end{equation*}
  On the other hand, since $P_2$ is the quadratic Taylor 
  polynomial for 
  $v_{2\de_{\pk}}(x),$ we should have
  
  \begin{gather*}%\label{63}
    |v_{2\de_{\pk}}(x)-P_2(x,\tb_{\pk})|\le c 
    \de_{\pk}^3\sup_{y\in\tb_{\pk}+B_{\de_{\pk}}(\mathbb{0}_{\Nb})}|\nabla^3(y)|\le 
    \\\nonumber
    c\mathbbm{c} 
    \de_{\pk}^3\de_{\pk}^{\s-1}=c\mathbbm{c}\de_{\pk}^{\s+2}, \, 
    x\in 
    \tb_{\pk}+B_{\de_{\pk}}(\mathbb{0}_{\Nb}). 
  \end{gather*}
  The latter two inequalities imply
  
  \begin{equation}\label{64}
    |P_2(x,t_{\pk})-|x|^2|\le  c\mathbbm{c}\de_{\pk}^{2+\s}.
  \end{equation}
  Since both polynomials  in \eqref{64} have degree 2 and 
  \eqref{64} must hold for 
  any $\de_{\pk}$, we have a contradiction:
  $P_2(x,t_{\pk})$ is harmonic, while $|x|^2$ is not. 
 \end{proof}

 \appendix
 \section{Estimates for  integrals in Section \ref{Sect.Integrals}. Proofs}
 \emph{Proof} of Lemma \ref{lem2}.\\
For $y\in B_{2\de}(x)$, we have $d(y)\le c\de,$ 
  therefore,
\begin{equation*}
  \int_{B_{2\de}(x)}\frac{d(y)^{\rb-2}\om(d(y))}{|x-y|^{\Nb-2}} 
  dy\le C 
  \de^{\rb}\int_{B_{2\de}(x)}\frac{d(y)^{-2}\om(d(y))}{|x-y|^{\Nb-2}} 
  dy.
\end{equation*}
For the last integral, we can use the estimate in \cite{RZ25}, see 
there Lemma 6.2, 
which gives us the required inequality.\hfill{$\square$}

\emph{Proof} of Corollary \ref{Cor.L2}\\
    If $y\in {B_{2\de}(x)}$, then $\de^{2-\Nb}\le C 
    |x-y|^{2-\Nb}$, therefore, 
    \eqref{21} implies
\begin{gather*}%\label{22.a}
 \int_{B_{2\de}(x)} d(y)^{\rb-2}\om(d(y))dy= 
 \de^{\Nb-2}\int_{B_{2\de}(x)} \frac{ 
 d(y)^{\rb-2}\om(d(y))}{\de^{\Nb-2}}dy\\\nonumber
\le C \de^{\Nb-2}\int_{B_{2\de}(x)} \frac{ 
d(y)^{\rb-2}\om(d(y))}{|x-y|^{\Nb-2}}dy\le c 
\de^{\Nb-2+\rb}\om(\de). 
\end{gather*}\hfill{$\square$}

\emph{ Proof} of Lemma \ref{lemA1} We set $\de_1=c_0\de$. The integral in 
   \eqref{int.4} can be 
   represented as $I_{k}=\sum_{\ka=1}^\infty I_{(\ka)}$, where 
   $I_{(\ka)}$ is the 
   integral over the spherical annulus 
   $\Tc_\ka=B_{2^{-\ka+1}\de_1}(x)\setminus 
   B_{2^{-\ka}\de_1}(x). $
For the latter integral, we have
\begin{gather}\label{A5}
  I_\ka\le c 
  \frac{1}{(2^{-\ka}\de_1)^{\Nb-2+k}}\int_{\Tc_{\ka}}d(y)^{k-2}\om(d(y))dy 
  \le
  \\\nonumber
  c\frac{1}{(2^{-\ka}\de_1))^{\Nb-2+k}}\int_{B_{2^{-\ka+1}\de_1}(x)}d(y)^{k-2}\om(d(y))dy\le\\\nonumber
  c\frac{1}{(2^{-\ka}\de_1))^{\Nb-2+k}} 
  (2^{-\ka}\de_1)^k\int_{B_{2^{-\ka+1}\de_1}(x)} 
  d(y)^{-2}\om(d(y))dy\le 
  \\\nonumber  
  c\frac{1}{(2^{-\ka}\de_1))^{\Nb-2+k}}\int_{B_{2^{-\ka+1}\de_1}(x)}d(y)^{k-2}\om(d(y))dy\le\\\nonumber
  c\frac{1}{(2^{-\ka}\de_1))^{\Nb-2+k}} (2^{-\ka}\de_1)^k 
  \int_{B_{2^{-\ka+1}\de_1}(x)}\frac{(2^{-\ka} 
  \de_1)^{\Nb-2}}{|x-y|^{\Nb-2}}d(y)^{-2}\om(d(y))dy\le\\\nonumber
 c\frac{1}{(2^{-\ka}\de_1))^{\Nb-2+k}} (2^{-\ka}\de_1)^k (2^{-\ka 
 \de_1})^{\Nb-2}\int_{B_{2^{-\ka+1}\de_1}(x)} 
 \frac{d(y)^{-2}\om(d(y))}{|x-y|^{\Nb-2}}dy\le\\\nonumber
\le c \om(2^{-\ka}\de_1)
\end{gather}
(the last inequality was established in  Lemma 6.2 in 
\cite{RZ25}).
We sum over $\ka$ the expressions in \eqref{A5}, which gives
\begin{gather*}%\label{A6}
I_k\le  c\sum_{\ka=1}^\infty \om(2^{-\ka}\de_1)\le 
C\int_1^\infty\om(2^{-\tau}\de_1)d\tau=\\\nonumber
c\int_{0}^{\frac12}\om(t\de_1)\frac{dt}{t}=c\int_0^{\de_1/2}\om(t)t^{-1}dt\le 
c\om({\de_1}/{2})\le c\om(\de).
\end{gather*}\hfill{$\square$}

  \emph{Proof} of Corollary \ref{Cor2}
   In fact, by Lemma \ref{lemA1}, we have   
   \begin{equation*}
     J_k\le c 
     \de^{\rb-k}\int_{B_{c_0\de}(x)}\frac{d(y)^{k-2}\om(d(y))}{|y-x|^{\Nb-2+k}}dy\le 
     c \de^{\rb-k}\om(\de).
   \end{equation*}\hfill{$\square$}
   
\emph{Proof} of Lemma \ref{Lemma2}.
Again, we represent the integral in \eqref{A7} as the 
sum of integrals 
$I_{\ka}$ over, now expanding,  spherical annuli 
$\Tc_{\ka}=B_{2^\ka 
\de(x)}\setminus B_{2^{\ka-1} \de(x)}$. For a single integral 
$I_\ka$, we have, 
similarly to how we derived \eqref{A5},
\begin{gather}\label{A8}
  I_{\ka}\le c\frac{1}{(2^{\ka}\de_1)^{\Nb-1+k}}\int_{\Tc_{\ka}} 
  d(y)^{k-2}\om(y)dy\le\\\nonumber
  c\frac{1}{(2^\ka\de_1)^{\Nb-1+k}}\int_{B^{2^{\ka}\de_1}(x)}d(y)^{k-2}\om(d(y))dy\le\\\nonumber
    c\frac{1}{(2^\ka\de_1)^{\Nb-1+k}} 
    (2^{\ka}\de_1)^{k+\Nb-2}\om(2^{\ka}\de_1).
\end{gather}
Next, \eqref{A8} implies
\begin{equation*}%\label{A9}
I_{(\ka)}\le c\frac{1}{2^{\ka}}\om(2^\ka\de_1).
\end{equation*}
Finally, similarly to the preceding calculations, we have
\begin{gather*}
  \sum_{\ka=1}^{\infty}I_{(\ka)}\le c 
  \sum_{\ka=1}^{\infty}\frac{1}{2^{\ka}\de_1}\om(2^\ka\de_1)\le 
  \int_1^{\infty}(2^{\tau}\de_1)^{-1}\om(2^{\tau}\de_1)d\tau=\\\nonumber
c\int_{2\de_1}^{\infty}t^{-1}\om(t)\frac{dt}{t}\le c 
\frac{\om(2\de_1)}{2\de_1}\le 
c\frac{\om(\de_1)}{\de_1}.
\end{gather*}\hfill{$\square$}

\emph{Proof} of Lemma \ref{Lem3a}.\\ Estimates follow directly from Lemmas \ref{lemA1} and 
\ref{lem2}.\hfill{$\square$}

\emph{Proof} of Lemma \ref{Lem4}. It goes over the same lines as the proof of Lemma 
\ref{Lemma2}.
The integral in \eqref{eq.Lem4} can be represented as
\begin{equation*}%\label{eq.Lem4.1}
  \Jc(x)= \sum_{j\ge1}\Jc_j(x):= 
  \sum_{j=1}^{\infty}\int_{\Ac_j}d(y)^{\rb-2}\om(d(y))|x-y|^{1-\Nb}dy,
\end{equation*}
where  $\Ac_j$ is the spherical annulus, 
\begin{equation*}
\Ac_j=B_{2^j\de}(x)\setminus B_{2^{j-1}\de}(x)
\end{equation*}
 For $j\ge1,$ $y\in \Ac_j(x)$, by  \eqref{43},
 we have $\le C |x-y|^{1-\Nb}\le \de (2^j\de)^{1-\Nb},$ therefore,
 
 \begin{equation*}%\label{B5}
  \Jc_j(x)\le c (2^j\de)^{1-\Nb}\int_{\Ac_j(x)}d(y)^{\rb-2}\om(d(y))dy.
 \end{equation*}
 Since the volume of $\Ac_j(x)$ is no greater than $C(2^j\de)^{\Nb-1}$,
 we have 
 \begin{gather}\label{B6}
   \Jc_j(x)\le c \de(2^j\de)^{1-\Nb}(2^j\de)^{\rb+\Nb-2}\om(2^j\de)=\\ \nonumber
   (2^j\de)^{\rb-1}\om(2^j\de).
 \end{gather}
The sum in \eqref{B6} is, in fact, finite. it contains only terms with $2^j\de\le 1$, 
therefore $\Jc_j(x)\le C (2^j\de)^{-1}$. After we sum the expression in \eqref{B6} over $j\ge 
1$, we obtain

  \begin{gather*}%\label{B8A}
\Jc(x)\le  \de^{-1} \sum_j 2^{-j}\om(2^j\de)\le\\\nonumber
   c\de^{-1}\int_0^{\infty}\frac{\om(2^t\de)}{2^t}dt=\\\nonumber 
   \de^{-1}c'\int_1^\infty\frac{\om(\tau\de)}{\tau^2}d\tau\le c''\de^{-1}\om(\de).
 \end{gather*}
\hfill{$\square$}
 \section{The Green function and its derivatives}\label{Sect.A1}
\subsection{General}
In this appendix, we establish some estimates for the Green 
function $G(x,y;\vs)$ of a 
second order 
elliptic operator and for its derivatives in the unit ball
$B\subset{\R^{\Nb}}$ with smooth boundary. 
Derivatives in the variables $x,y$ are well studied; the estimates 
for first  and second derivatives, needed for 
our applications, 
are contained in the classical paper \cite{Widman}. We, however, need estimates for derivatives of a 
higher
order; moreover, in Sect. \ref{Sect.K(x,y)} we
consider  operators whose coefficients depend  on an extra 
parameter $\vs$ in a small ball  $\Bc\subset\R^\Nb$, and we need estimates for derivatives of the 
Green function in 
all three variables. 

The proofs will be presented further on in 
this Appendix. Since we consider only the unit ball, we omit the 
superscript $\circ$ in the notation of the Green function.

\subsubsection{Schauder estimates}

Note that the estimates of derivatives not involving $\vs$ are 
already contained in \cite{Kras.Sing}, see Theorem \ref{Kras.thm}. So, 
it is only the derivatives    $G_\vs, G_{x\vs}, G_{\vs\vs}, G_{xx\vs}, 
G_{xy\vs}, G_{x\vs\vs}$ that we need to consider, with points $x,y$ well 
separated. Namely, the point $x$ should be in a small neighborhood 
of the centerpoint $\mathbb{O}$  of the ball, while $y$ 
should be near some point $y^{\circ}$ on the boundary of the  ball. Thus, the singularity of the
Green function at $x=y$ is cut-away. 

The result we are going to use systematically is a consequence of 
Schauder estimates. We formulate the particular cases of interest, 
for a ball $\Bc_r$ of radius $r<1$, derivatives of order 2 ($k=2$).

\begin{thm}\label{Thm.Schauder}$[$Interior estimate$]$
 Let $\Lc$ be a second order elliptic  operator in $B_r$, with 
 coefficients in the H{\"o}lder class $C^{k+\g}$, $0<\g<1$ and let 
 $f\in C^{k+\g}(B_{\theta r}), 0<g<1, \, \theta<1$. 
 Then for the solution $u(x)$ of the equation $\Lc u=f$ in $\Bc_r$
the estimate holds
\begin{equation}\label{Schauder3}
  \pmb{|}u\pmb{|}_{2+k+\g,B_{\theta r}}\le \Cb 
  \left(\pmb{|}f\pmb{|}_{k+\g,B_r}+\|u\|_{C(\Bc_r)}\right),
\end{equation}
where the notations as $ \pmb{|}u\pmb{|}_{\g,B_r}$ denote the norm of $u$ in the 
H{\"o}lder class $C^{\g, B_r}$ etc., and the constant in 
\eqref{Schauder3} depends on the ellipticity constant of the 
operator $\Lc$, and H{\"o}lder norms of the coefficients and the parameter $\theta<1.$
\end{thm}
Thus, this interior Schauder estimate states that the solution of a 
second order elliptic equation has, in the H{\"o}lder scale, higher, by 
order 2, smoothness than the right-hand side, with norm controlled by 
the corresponding norm of $f$ in a larger ball.  For the proof, see 
\cite{GT}, Theorem 6.2.

We also need the boundary Schauder estimate. We formulate it for  the 
special case   when the domain is the intersection of two balls. Let 
$B$ be the unit ball and $B_\e$ be a ball of radius $\e,$ 
centered at a point $y^{\circ}\in \G\equiv\partial B$. Denote by $U_\e$ the 
intersection $U_\e=B\cap B_\e$ and by $U_{\e/2}$ the intersection of $B$ 
with the smaller ball with center $y^{\circ}$.
\begin{thm}\label{boundary Schuder}$[$Boundary estimate$]$ Let $u(x)$ be a 
solution of the equation $\Lc u=f$ in $\Dc=U_\e$ with $f\in C^{\g}$, such 
that $u(y)=0$ on $\G\cap B_{r}$. Then 
\begin{equation*}%\label{Schauder02}
  \pmb{|}u\pmb{|}_{k+2+\g, U_{\e/2}}\le C\left(
  \pmb{|}f\pmb{|}_{k+\g,\Dc}+\|u\|_{C(B\cap U_\e)}\right).
\end{equation*}
\end{thm} This result follows immediately from the  main local step in 
the proof of the standard boundary Schauder estimate for a domain having 
a piece of hyperplane as a part of boundary, by means of a change of 
variables straightening the boundary, see, e.g., \cite{GT}, Corollary 
6.7.

\subsection{Proof of Theorem \ref{Th.A1}}\label{Proof of GF estimates}
\begin{proof}
The proof consists of two steps. On the first step, we establish 
estimates for the first and second order derivatives of the Green 
function in $\vs$ variable; here we use Krasovskii's estimates. On the second 
step, by using Schauder estimates we perform a kind of bootstrap and 
derive estimates for derivatives involving variables $x$ and $y$ as 
well. Below, $\partial_{\vs}$ denotes the derivative with respect to any of variables $\vs$.
  
  We start by finding estimates for derivatives of the 
Green 
function $G(x,y;\vs)$ in $\vs$ variable.
We have $\Lc(\vs)\Gb(\vs)=\Ib$, the latter is the identity operator. 
We differentiate this equality in $\vs$ variable;
for the derivative or order  one, we have
\begin{equation}\label{AA1}
\partial_\vs \Gb(\vs)=-\Gb(\vs)\Lc_\vs\Gb(\vs);\, \Lc_\vs=\partial_\vs\Lc,
\end{equation}
the latter is a second order operator with coefficients 
$\ab_\vs=\partial_{\vs}\ab({\vs})$ in $C^{m-1}.$
For the Green function, \eqref{AA1} gives
\begin{equation}\label{AA2}
  \partial_{\vs}G(x,y;\vs)=-\int_{\Bc}G(x,z;\vs)\Lc_\vs(z,\partial_z;\vs)G(z,y;\vs) 
  dz. 
\end{equation}
    We integrate in \eqref{AA2} by parts; the boundary term 
    vanishing due to the Dirichlet boundary conditions for $G$, so
    
    \begin{equation}\label{AA3}
     \partial_{\vs}G(x,y;\vs)=\int_{\Bc}\langle \ab_\vs(z)\nabla_z 
     G(x,z;\vs),\nabla_zG(z,y;\vs)\rangle dz.
    \end{equation}
   For estimating the integral in \eqref{AA3}, we use 
   \eqref{WidmanEst}, which gives
      \begin{equation*}%\label{AA4}
     | \partial_{\vs}G(x,y;\vs)|\le 
     C\int_{\Om}|x-z|^{1-\Nb}|z-y|^{1-\Nb}dz\le C |x-y|^{2-\Nb}\le C,
   \end{equation*}
   since $|x-y|>r_0>0.$
   
   With some more complications, we handle the second derivatives of 
   $G$.
   We differentiate \eqref{AA2} again, which gives

   \begin{gather}\label{AA5}
    \partial^2_{\vs\vs}G(x,y;w)=-\int_{\Bc}G(x,z;\vs)\Lc_{\vs\vs}(z,\partial_z;\vs)G(z,y;\vs) 
    dz\\\nonumber
     +2\int_{\Bc}\int_{\Bc}G(x,z_1;\vs)\Lc_\vs(z_1,\partial_{z_1},\vs)G(z_1,z_2;\vs)\Lc_\vs(z_2,\partial_{z_2};\vs)G(z_2,y;\vs)dz_1dz_2.
   \end{gather}
     Having $|x-y|>\frac34$, we consider three domains in 
     $\Bc\times\Bc$:\\ $\Dc_1=\{(z_1,z_2): |z_1-z_2|>\frac18\},$ \\
     $\Dc_2=\{(z_1,z_2):|z_1-z_2|<\frac14, |z_1-x|>\frac18\},$ \\
     $\Dc_3=\{(z_1,z_2):|z_1-z_2|<\frac14, |z_2-y|>\frac18\}.$
      
  These three domains, obviously, cover $\Bc\times\Bc$. We 
  introduce the decomposition of unity, subordinated to this 
  cover; $\psi_s(z_1,z_2),$ $s=1,2,3$, $\psi_s\in 
  C^{\infty}(\overline{\Bc}\times\overline{\Bc})$, $\psi_s=0$ 
  outside $\overline{\Dc_s}$, $\sum_{s}\psi_s(z_1,z_2)=1$. Order 1 
  and 2  derivatives of $\psi_s$ are bounded. Correspondingly, the 
  integral in \eqref{AA5} splits into three terms, with factors 
  $\psi_s(z_1,z_2), \, s=1,2,3$. We consider each of them.

  For $(z_1,z_2)\in \Dc_1,$ we integrate once by parts both in 
  $z_1$ and $z_2$ variables. We omit terms where the derivatives 
  fall on the cut-off function or on the coefficients of $\Lc$, 
  since these terms have a weaker singularity, and consider the 
  most singular terms, which typically have the form
  \begin{equation}\label{AA6}
    \Mc_1(x,y)=\int_{\Dc_1}\psi_1(z_1,z_2)\partial_{z_1}G(x,z_1;\vs)\partial_{z_1}\partial_{z_2}G(z_1,z_2;\vs)\partial_{z_2}G(z_2,y;\vs)dz_1dz_2,
  \end{equation}
   with bounded coefficients. Due to the estimate \eqref{WidmanEst}, the middle term in \eqref{AA6} is bounded since 
   $|z_1-z_2|>\frac18$ in $\Dc_1$, while two other terms have singularity 
   $O|x-z_1|^{1-\Nb},$ $O|x-z_1|^{1-\Nb},$ with coefficients 
   controlled by order 3 derivatives of the coefficients of 
   operator $\Lc.$ This implies that the term in \eqref{AA5}, 
   corresponding to $\psi_1$, is bounded.
   
   Next, we consider the integral over $\Dc_2.$ We integrate by 
   parts twice in $z_1$ variable and once in $z_2$ variable. 
   Again, omitting less singular terms, where the derivatives fall 
   on the cut-off function or coefficients, we arrive at the 
   integrals of the form
   
   \begin{equation}\label{AA7}
     \Mc_2(x,y)=\int_{\Dc_2}\psi_2(z_1,z_2)\partial_{z_1}^2G(x,z_1;\vs)\partial_{z_2}G(z_1,z_2;\vs)
     \partial_{z_2}G(z_2,y;\vs)dz_1dz_2.
   \end{equation}
       The first factor in the integrand   in \eqref{AA7} is 
       bounded, since here $|x-z_1|>\frac18$, while the other 
       terms, containing only order one derivatives, have 
       singularities of order $1-\Nb,$ thus, $ \Mc_2(x,y)$ is 
       bounded.
       
   Finally, we consider the integral over $\Dc_3.$ Here the last 
   factor in \eqref{AA5} is bounded. We integrate by parts once in 
   the second factor, and, ignoring again the terms where the 
   derivative falls on the cut-off function or coefficients, we 
   have, as the most singular term
   \begin{equation}\label{AA8}
     \Mc_3(x,y)=\int_{\Dc_3}\psi_3(z_1,z_2)\partial_{z_1}G(x,z_1;\vs)\partial_{z_1}G(z_1,z_2;\vs)\partial^2_{z_2}G(z_2,y;\vs)dz_1dz_2.
   \end{equation}
   Here, the last factor is bounded and the first and the second 
   one have singularity of order no greater that $1-\Nb$; all these 
   terms have coefficients depending on derivatives of coefficients of 
   order not higher than 3. Therefore, all three expressions 
   \eqref{AA6},\eqref{AA7},\eqref{AA8} are bounded.
   
   Next we pass to the study of derivatives of the Green function 
   in $x$ and $y$ variables.
   
   With $y$ fixed, and $x$ near the center of $B$, we consider 
   the balls $B_r=B_{r}(\mathbb{0})$ with radius $r$ less than $3/4$ and $B_{\theta r}=B_{\theta r}(\mathbb{0})$. For 
   $x\in B_r,$ we consider the functions $u_0(x)=G(x,y;\vs)$, 
   $u_1(x)=\partial_\vs G(x,y;\vs)$, $u_2(x)=\partial^2_{\vs\vs}G(x,y;\vs).$
   
   As it is already proved, the functions $u_0,u_1,u_2$ are 
   bounded in $B_r$. We also know that $u_0(x)\in C^3(\Dc).$ 
   After differentiating in $\vs$ the equality $\Lc(\vs)G(x,y;\vs)=0$, 
   $x\in B_r,$ we see that 
    $u_1(x)$  is  a solution of the equation $\Lc u_1=f_1$, with 
    $f_1(x)=-(\partial_\vs \Lc)u_0$. Thus, $f_1(x)\in 
    C^{1,\g}(B_{r/2}),$ moreover, $f_1(x)=0$ on $\partial \Bc.$
    We apply the Schauder estimate to the equation $\Lc u_1=f_1$, 
    which gives an estimate for the $x$-derivatives of $u_1,$   
    \begin{equation}\label{AD9}
      \pmb{|}u_1\pmb{|}_{3+\g,B_{\theta r}}\le C 
      \left(\|u_1\|_{L_{\infty}(B_r)}+\pmb{|}f_1\pmb{|}_{1+\g,B_r}\right).
    \end{equation}
   The inequality \eqref{AD9} shows that the derivatives in $x$, 
   up to the third order, are bounded in $B_{\theta r}$, in fact, they 
   are order 2 H\"older better than $f_1$.
   
   In a similar way, we consider the function $u_2(x)$. It is a 
   solution of the equation
   
   \begin{equation}\label{AD10}
     \Lc(\vs) u_2(x)=f_2(x)\equiv 
     -[(\partial_{\vs\vs}^2\Lc(\vs)]u_0-2[\partial_\vs\Lc(\vs)]u_1.
   \end{equation}
   We have just proved that $u_1\in C^{2+k+\g}(B_{\theta r})$; so, by our 
   assumptions on $\Lc$, we have $f_2\in C^{k+\g}(B_{\theta^2 r}) $.  Again, 
   applying Schauder's estimates, we obtain an inequality, similar to \eqref{AD9}, for 
   second derivatives of $u_2$.
   
   In order to handle  derivatives involving the differentiation in $y$, we just repeat the reasoning above for the function $v_0(y)=G(x,y;\vs)$ for a fixed $x\in B_r$ near the centerpoint $\mathbb{0}$ as a function of  $y\in U_{(\e)}=B_\e(y^{\circ})\cap B.$ This function is a solution of the equation $\Lc({y,\partial_y;\vs})v_0(y)=0,$ and after differentiation in $\vs,$ we arrive at the equation  $\Lc{(y,\partial_y;\vs)}\partial_{\vs}v_0=-\left[\partial_{\vs}\Lc\right]v_0\equiv g_1. $ Now we can apply the boundary Schauder estimate, since $\partial_{\vs}v_0=0$ on the boundary, to obtain the estimate for $v_1=\partial_{\vs}v_0$,
      
   \begin{equation}\label{N1}
  \pmb{|}v_1\pmb{|}_{2+\g,U_{(\e/2)}}\le C(\|v_1\|_{L_\infty(U_\e) }+  \pmb{|}g_1\pmb{|}_{\g,U_\e}).
   \end{equation}
The function $g_1$ on the right-hand side   in \eqref{N1} is bounded, uniformly in $x\in B_{r}(\mathbb{0}),$ moreover, has bounded derivatives in $x$, by Theorem \ref{Kras.thm}, since $x,y$ are separated. This gives the required estimate for $v_1$.

 Finally, we consider the mixed derivative, 
   $G_{\vs xy}(x,y;\vs)=\partial_x\partial_y\partial_\vs G(x,y;\vs)$ of the Green 
   function, for $x\in B_{r}$ and $y\in U_\e$
   
   We apply again Schauder estimates, this time in the following setting.
 The function $v_2(x)=\partial_y G(x,y,\vs)$ satisfies 
   $\Lc(x,\partial_x;\vs) v_2=0$. We differentiate the latter equality in $\vs$ and for $v_3=\partial_{\vs}v_2$ obtain the equation
     \begin{equation}\label{N2}
   \Lc(x,\partial_x;\vs)_x v_3= g_2=-[\partial_{\vs}\Lc(x,\partial_x;\vs)] 
   \end{equation}
  The right-hand side in \eqref{N2} belongs to $C^{k-2+\g}$ uniformly in $y$ variable in $U_\e$, therefore, $v_3$ belongs to $C^{k+\g}$ in $x$ variable  which implies the required estimate
   \begin{equation*}
    G_{\vs xy}(x,y;\vs)\le C.
  \end{equation*}
In a similar way, we can establish this kind of estimates for higher order derivatives in $x,y$,
provided the coefficients are sufficiently smooth.
 
   \end{proof}

 \end{document}